\documentclass[10pt]{article}

\usepackage{amsmath,amsthm,amsfonts,amssymb,bm,bbm,enumerate, graphicx, mathtools,color}
\usepackage{tcolorbox}
\usepackage[english]{babel}
\usepackage[utf8]{inputenc}
\usepackage{fancyhdr}
\usepackage{caption}
\usepackage{subcaption}
\usepackage[toc,page]{appendix}
\usepackage{csquotes,multirow}

\definecolor{NBrown}{HTML}{66220C}
\definecolor{NAqua}{HTML}{00698C}
\definecolor{ForestGreen}{HTML}{228b22}

\usepackage{hyperref}
\usepackage[capitalise]{cleveref}

\parskip \medskipamount

\usepackage{mleftright}
\mleftright

\newcommand{\sqn}{\sqrt{n}}

\newtheorem{theorem}{Theorem}[section]
\newtheorem{lemma}[theorem]{Lemma}

\newtheorem{corollary}[theorem]{Corollary}
\newtheorem{remark}[theorem]{Remark}
\newtheorem{claim}[theorem]{Claim}
\newtheorem{assumption}[theorem]{Assumption}
\newtheorem{definition}[theorem]{Definition}

\newtheorem{example}[theorem]{Example}
\newtheorem{open}[theorem]{Problem}

\newcommand{\E}[0]{\mathbb{E}}
\newcommand{\pr}[0]{\mathbb{P}}
\newcommand{\Var}[0]{\text{Var}}
\newcommand{\prstart}[2]{\mathbb{P}_{#2}\!\left(#1\right)}
\DeclareMathOperator{\var}{Var}

\DeclareMathOperator{\CRT}{CRT}
\DeclareMathOperator{\LE}{LE}

\DeclareMathOperator{\LERW}{LERW}
\DeclareMathOperator{\Capa}{Cap}

\DeclareMathOperator{\UST}{UST}

\DeclareMathOperator{\diam}{Diam}
\DeclareMathOperator{\height}{Height}

\renewcommand{\epsilon}{\varepsilon}
\newcommand{\pand}[0]{\ \text{and} \ }

\newcommand{\RR}{\mathbb{R}}

\newcommand{\ZZ}{\mathbb{Z}}
\newcommand{\XX}{\mathbb{X}}
\newcommand{\NN}{\mathbb{N}}

\newcommand{\T}{\mathcal{T}}

\newcommand{\M}{\mathcal{M}}
\newcommand{\Mm}{M_\ell}
\newcommand{\Wj}{\varphi_j}
\newcommand{\V}{V^{\text{bad}}}

\newcommand{\Capp}{\mathrm{Cap}}
\newcommand{\Capk}{\mathrm{Cap}_k}

\newcommand{\Capkl}{\mathrm{Cap}_{k_\ell}}

\newcommand{\tmix}{t_{\mathrm{mix}}}

\newcommand{\close}{\mathrm{Close}}

\newcommand{\Cavoidi}{X^{\mathbf{avoid}}_{I}}

\newcommand{\pvec}[1]{\vect{p}_{#1}}
\newcommand{\vect}{\mathbf}

\newcommand{\eps}{\varepsilon}

\newcommand{\pcomment}[3]{\textcolor{#1}{\textit{\{#2: #3\}}}}
\newcommand{\ellie}[1]{\pcomment{NAqua}{Ellie}{#1}}

\newcommand{\ru}{u}
\newcommand{\rv}{v}

\newcommand{\supp}{\mathrm{supp}}
\newcommand{\dGHP}{d_{\mathrm{GHP}}}
\newcommand{\dGP}{d_{\mathrm{GP}}}
\newcommand{\GHP}{\mathrm{GHP}}
\newcommand{\GP}{\mathrm{GP}}
\newcommand{\K}{\chi}

\newcommand{\dumr}{r_{\ell}}
\newcommand{\dumep}{\eps_{\ell}}
\newcommand{\dumk}{k_{\ell}}

\newcommand{\fp}{\zeta}
\newcommand{\secp}{\psi}
\newcommand{\Engood}{\mathcal{E}_{j}}
\newcommand{\EngoodBuf}{\mathcal{E}^{\mathrm{ind}}_{j}}
\newcommand{\Ijbuf}{Y_j}
\newcommand{\Enc}{\mathcal{E}_{n,c,\eps}}
\newcommand{\sunk}{\odot}
\newcommand{\Bsun}{B^{\sunk}}
\newcommand{\Bj}{B^{\sunk}_j}
\newcommand{\Bjj}{B^{\sunk}_{j-1}}
\newcommand{\Bz}{B^{\sunk}_0}
\newcommand{\Gj}{G_n(\Wj)}

\usepackage{tikz}

\usepackage{graphicx}
\usepackage{caption,subcaption}

\makeatletter
\renewcommand\@dotsep{10000}
\makeatother
\allowdisplaybreaks

\begin{document}

\title{The GHP scaling limit of uniform spanning trees in high dimensions}
\author{Eleanor Archer, Asaf Nachmias, Matan Shalev} 

\date{\today}
	\maketitle
	\begin{abstract} We show that the Brownian continuum random tree is the Gromov-Hausdorff-Prohorov scaling limit of the uniform spanning tree on high-dimensional graphs including the $d$-dimensional torus $\ZZ_n^d$ with $d>4$, the hypercube $\{0,1\}^n$, and transitive expander graphs. Several corollaries for associated quantities are then deduced: convergence in distribution of the rescaled diameter, height and simple random walk on these uniform spanning trees to their continuum analogues on the continuum random tree. 
	\end{abstract}
	
	\section{Introduction} Consider the uniform spanning tree (UST) of the $d$-dimensional torus $\ZZ_n^d$ with $d>4$ or another transitive high-dimensional graph such as the hypercube $\{0,1\}^n$ or a transitive expander graph. In this paper we show that the Brownian continuum random tree (CRT), introduced by Aldous \cite{AldousCRTI, AldousCRTII}, is the Gromov-Hausdorff-Prohorov (GHP) scaling limit of such USTs.

	Convergence of such USTs to the CRT in the sense of finite dimensional distributions has been established in the work of Peres and Revelle \cite{PeresRevelleUSTCRT}. The novelty of the current paper is proving that this convergence holds in the stronger GHP topology. This implies the convergence in distribution of some natural geometric quantities of the USTs (which were not known to converge prior to this work) and allows us to express their limiting distribution explicitly. For example, it follows from our work that the diameter and the height seen from a random vertex of these USTs, properly rescaled, converge to certain functionals of the Brownian excursion, as predicted by Aldous (see \cite[Section 4]{AldousCRTII}). Additionally, it implies that the simple random walk on these USTs converges to Brownian motion on the CRT. We discuss these implications in Section \ref{sctn:corollaries of main thm}.

Our main result is as follows.	

	\begin{theorem} \label{thm:main1} Let $\T_n$ be a uniformly drawn spanning tree of the $d$-dimensional torus $\ZZ_n^d$ with $d>4$. Denote by $d_{\T_n}$ the corresponding graph-distance on $\T_n$ and by $\mu_n$ the uniform probability measure on the vertices of $\T_n$. Then there exists a constant $\beta(d)>0$ such that
		\begin{equation} \label{eq:mainthm1} \left(\T_n,\frac{d_{\T_n}}{\beta(d) n^{d/2}} ,\mu_n\right) \overset{(d)}{\longrightarrow} (\T,d_{\T},\mu)\end{equation}
		where $(\T,d_{\T},\mu)$ is the CRT equipped with its canonical mass measure $\mu$ and $\overset{(d)}{\longrightarrow}$ means convergence in distribution with respect to the GHP distance between metric measure spaces.
	\end{theorem}
	
	\begin{remark}
		We take the convention of Aldous \cite[Section 2]{AldousCRTII} that the CRT is coded by two times standard Brownian excursion, although different normalizations are sometimes used elsewhere in the literature.
	\end{remark}

	Our result shows that high-dimensional USTs exhibit a strong form of \emph{universality}, a common phenomenon in statistical physics whereby above an \emph{upper critical dimension}, the macroscopic behaviour of a system does not depend on the finer properties of the underlying network. For USTs the upper critical dimension is well-known to be four as for the closely related model of \emph{loop-erased random walk} (LERW). Above  dimension four LERW rescales to Brownian motion, see \cite{LawlerIntersectionsBook}. In lower dimensions the scaling limits are markedly different. On $\ZZ^2$ it was shown by Lawler, Schramm and Werner \cite{lawler2011conformal} that LERW rescales to $\mathrm{SLE}_2$, and Barlow, Croydon and Kumagai \cite{BCK20172d} later established subsequential GHP scaling limits for the UST. This was later extended to full convergence in a result of Holden and Sun \cite{holden2018sle}.
	On $\ZZ^3$, much less is known, however the breakthrough works of Kozma \cite{kozma2007scaling} and Li and Shiraishi \cite{li2018convergence} on subsequential scaling limits of LERW enabled Angel, Croydon, Hernandez-Torres and Shiraishi \cite{ACHTS20203d} to show GHP convergence of the rescaled UST along a dyadic subsequence. Their scaling factors are given in terms of the LERW growth exponent in three dimensions, which was shown to exist by Shiraishi \cite{Shir18}. Finally, in four dimensions, a classical result of Lawler \cite{LawlerLERW4d} computes the logarithmic correction to scaling under which the LERW on $\ZZ^4$ converges to Brownian motion. Schweinsberg \cite{schweinsberg} showed that with these logarithmic corrections to scaling, the finite-dimensional distributions of the UST on the four dimensional torus converge to those of the CRT, analogously to \cite{PeresRevelleUSTCRT}. Various exponents governing the shape of the UST in $\ZZ^4$ are given in the recent work of Hutchcroft and Sousi \cite{hutchcroftsousi}. Our proof of GHP convergence does not encompass the four dimensional torus (see \cref{open:4d}).

	In the rest of this section we first present the standard notation and definitions required to parse \cref{thm:main1}. We then state the most general version of our result, \cref{thm:maingeneral}, handling other high-dimensional underlying graphs such as expanders and the hypercube. We close this section with a discussion of the various corollaries mentioned above and the organization of the paper.

	\subsection{Standard notation and definitions}\label{sec:defs}
	
	A {\bf spanning tree} of a connected finite graph $G$ is a connected subset of edges touching every vertex and containing no cycles. The {\bf uniform spanning tree} (UST) is a uniformly drawn sample from this finite set. Given a tree $T$ we denote by $d_T$ the graph distance metric on the vertices of $T$, i.e., $d_T(u,v)$ is the number of edges in the unique path between $u$ and $v$ in $T$.
	
	We follow the setup of \cite[Sections 1.3 and 6]{MiermontTessellations} and work in the space $\XX_c$ of equivalence classes of (deterministic) metric measure spaces (mm-spaces) $(X,d,\mu)$ such that $(X,d)$ is a compact metric space and $\mu$ is a Borel probability measure on $(X,d)$, where we treat $(X,d,\mu)$ and $(X',d',\mu')$ as equivalent if there exists a bijective isometry $\phi: X \to X'$ such that $\phi_* \mu = \mu'$ where $\phi_*\mu$ is the pushforward measure of $\mu$ under $\phi$. As is standard in the field, we will abuse notation and represent an equivalence class in $\XX_c$ by a single element of that equivalence class. 
	
	We will now define the GHP metric on $\XX_c$. First recall that if $(X, d)$ is a metric space, the {\bf Hausdorff distance} $d_H$ between two sets $A, A' \subset X$ is defined as
	\[
	d_H(A, A') = \max \{ \sup_{a \in A} d(a, A'), \sup_{a' \in A'} d(a', A) \}.
	\]
	Furthermore, for $\epsilon>0$ and $A\subset X$ we let $A^{\epsilon} = \{ x \in X: d(x,A) < \epsilon \}$ be the $\epsilon$-fattening of $A$ in $X$. If $\mu$ and $\nu$ are two measures on $X$, the {\bf Prohorov distance} between $\mu$ and $\nu$ is given by
	\[
	d_P(\mu, \nu) = \inf \{ \epsilon > 0: \mu(A) \leq \nu(A^{\epsilon}) + \epsilon \text{ and } \nu(A) \leq \mu(A^{\epsilon}) + \epsilon \text{ for any closed set } A \subset X \}.
	\]
	
	\begin{definition} \label{def:GHP} Let $(X,d,\mu)$ and $(X',d',\mu')$ be elements of $\XX_c$. The \textbf{Gromov-Hausdorff-Prohorov} {\rm (GHP)} distance between $(X,d,\mu)$ and $(X',d',\mu')$ is defined as
		\[
		\dGHP((X,d,\mu),(X',d',\mu')) = \inf \left\{d_H (\phi(X), \phi'(X')) \vee d_P(\phi_* \mu, \phi_*' \mu') \right\},
		\]
		where the infimum is taken over all isometric embeddings $\phi: X \rightarrow F$, $\phi': X' \rightarrow F$ into some common metric space $F$.
	\end{definition}
	It is shown in \cite[Theorem 6 and Proposition 8]{MiermontTessellations} that $(\XX_c, \dGHP)$ is a Polish metric space. Denote by $\M_1(\XX_c^{\GHP})$ the set of probability measures on $(\XX_c,\dGHP)$ with the Borel $\sigma$-algebra. We say that a sequence of probability measures $\{\pr_n\}_{n=1}^\infty \subset \M_1(\XX_c^{\GHP})$ {\bf converges in distribution} to $\pr \in \M_1(\XX_c^{\GHP})$ if for any bounded continuous function $f:(\XX_c, \dGHP) \to \RR$ we have $\lim_{n} \E_n f = \E f$, where $\E_n$ and $\E$ are the expectation operators corresponding to $\pr_n$ and $\pr$. As usual, if $\{X_n\}$ and $X$ are random variables taking values in $(\XX_c, \dGHP)$, we say that $X_n \overset{(d)}{\longrightarrow} X$ if the corresponding pushforward measures of $X_n$ converge in distribution to that of $X$.
	
	The CRT is a typical example of a random fractal tree and can be thought of as the scaling limit of critical (finite variance) Galton-Watson trees. As we shall explain in Section \ref{sec:high-level}, we do not directly approach the CRT in this paper; therefore we have opted to omit the definition of the CRT and refer the reader to Le Gall's comprehensive survey \cite{legall2005randomtreesandapplications} for its construction (see also \cite{AldousCRTII}) as a random element in $\XX_c$. Except for this, by now we have stated all the necessary definitions required for \cref{thm:main1}.

	\subsection{The general theorem}
	We now present the general version of \cref{thm:main1} which will imply the GHP convergence of the UST on graphs like the hypercube $\{0,1\}^m$ or transitive expanders. Our assumptions on the underlying graph are stated in terms of random walk behavior but should be thought of as geometric assumptions. For a graph $G$, two vertices $x,y$ and a non-negative integer $t$ we write $p_t(x,y)$ for the probability that the lazy random walk starting at $x$ will be at $y$ at time $t$. When $G$ is a finite connected regular graph on $n$ vertices we define the {\bf uniform mixing time} of $G$ as
	\begin{equation} \label{def:tmix}
	\tmix(G) = \min\left\{t\geq 0: \max_{x,y\in G} \left| n p_t(x,y)  - 1\right| \leq \frac{1}{2}\right\}, 
	\end{equation}
	We will assume the following throughout the paper. This is the same assumption under which Peres and Revelle establish finite-dimensional convergence in \cite{PeresRevelleUSTCRT}.

	\begin{assumption}\label{assn:main}
		Let $\left\{ G_n \right\}$ be a sequence of finite connected vertex transitive graphs with $|G_n| = n$.  
		\begin{enumerate}
			\item There exists $\theta < \infty$ such that $\displaystyle \sup_n \sup_{x \in G_n} \sum_{t=0}^{\sqrt{n}} (t+1) p_t(x,x) \leq \theta$;
			\item There exists $\alpha > 0$ such that $\tmix (G_n) = o (n^{\frac{1}{2} - \alpha})$ as $n \to \infty$.
		\end{enumerate}
	\end{assumption}
	
Both items in \cref{assn:main} imply that the graph sequence is in some sense of dimension greater than four. The first item is a finite analogue of the condition that the expected number of intersections of two independent random walks is finite; in $\ZZ^d$ this happens if and only if $d>4$. The second item (which clearly holds on the torus on $n$ vertices once $d>4$, since this has mixing time of order $n^{2/d}$) heuristically ensures that different parts of the $\UST$ that are distance $\sqrt{n}$ apart behave asymptotically independently. We do not claim that these conditions are optimal (see the discussion in \cite[Section 1.4]{MNS}), but they are enough to yield convergence to the CRT in the most interesting cases.

	\begin{theorem}\label{thm:maingeneral}
		Let $\{ G_n \}$ be a sequence of graphs satisfying \cref{assn:main} and let $\T_n$ be a sample of $\UST(G_n)$. Denote by $d_{\T_n}$ the graph distance on $\T_n$ and by $\mu_n$ the uniform probability measure on the vertices of $\T_n$. Then there exists a sequence $\{\beta_n\}$ satisfying $0<\inf_n \beta_n \leq \sup_n \beta_n < \infty$ such that  
		\begin{equation*}
		\left(\T_n,\frac{d_{\T_n}}{\beta_n \sqrt{n}}, \mu_n \right) \overset{(d)}{\longrightarrow} \left( \T, d_{\T}, \mu \right)
		\end{equation*}
		where $(\T,d_{\T},\mu)$ is the {\rm CRT} equipped with its canonical mass measure $\mu$ and $\overset{(d)}{\longrightarrow}$ means convergence in distribution with respect to the {\rm GHP} distance.
	\end{theorem}

The sequence $\{\beta_n\}$ is inherited from the main result of Peres and Revelle, see \cite[Theorem 1.2]{PeresRevelleUSTCRT} (we restate this as \cref{thm:pr04} in this paper). 
Note that \cref{thm:main1} is not a special case of \cref{thm:maingeneral} since the latter does not guarantee a single scaling factor $\beta$, rather a sequence $\beta_n$ (which is the best one can hope for in the context of \cref{thm:maingeneral} since one can alternate between different graph sequences). 

\begin{proof}[Proof of \cref{thm:main1} given \cref{thm:maingeneral}]
For the torus $\ZZ_n^d$ with $d\geq 5$, Peres and Revelle proved that there exists $\beta(d)\in(0,\infty)$ such that \cite[Theorem 1.2]{PeresRevelleUSTCRT} holds with $\beta_n=\beta(d)$; see the choice of $\beta_n$ at the end of Section 3 of \cite{PeresRevelleUSTCRT} as well as Lemma 8.1 and (17) in that paper. Hence, this and \cref{thm:maingeneral} readily imply \cref{thm:main1}.
\end{proof}

Furthermore, see Lemma 1.3 and Section 9 of \cite{PeresRevelleUSTCRT}, in graphs where additionally two independent simple random walks typically avoid one another for long enough (see the precise condition in \cite[Equation 6]{PeresRevelleUSTCRT}), we can take $\beta_n \equiv 1$. This family of graphs includes the hypercube and transitive expanders with degrees tending to infinity. In the same spirit, for the $d$-dimensional torus, $\beta(d) \to 1$ as $d\to\infty$. Moreover, it is also immediate to see that \cref{assn:main} holds for a sequence of bounded degree transitive expanders (see for instance \cite[Section 9]{PeresRevelleUSTCRT}) and hence \cref{thm:maingeneral} holds for them as well.

	\subsection{Corollaries}\label{sctn:corollaries of main thm}
\subsubsection{Pointed convergence}
In order to establish some of the corollaries alluded to above, it will be useful to rephrase \cref{thm:maingeneral} in terms of \textit{pointed convergence}. Roughly speaking, this means that we consider our spaces to be rooted, and we add a term corresponding to the distance between the roots in the embedding in \cref{def:GHP}. We refer to \cite[Section 2.2]{Croydon18Scaling} for the precise definition. We start with the following observation, which is a trivial consequence of a coupling characterization of the Prohorov distance (see \cite[Proof of Proposition 6]{MiermontTessellations}).

\begin{lemma}\label{lem:pointed GHP conv}
Suppose that $(X_n, d_n, \mu_n)_{n \geq 1}, (X, d, \mu)$ are in $\XX_c$ with $(X_n, d_n, \mu_n) \to (X, d, \mu)$ deterministically in the GHP topology. Let $U_n$ be a random element of $X_n$ sampled according to the measure $\mu_n$, and $U$ be a random element of $X$ sampled according to the measure $\mu$. Then
\[
(X_n, d_n, \mu_n, U_n) \overset{(d)}{\to} (X, d, \mu, U)
\]
with respect to the \textit{pointed} GHP topology, as defined in \cite[Section 2.2]{Croydon18Scaling}.
\end{lemma}

Due to transitivity, in our setting the root can be an arbitrary vertex $O_n$ rather than uniformly chosen. Combining \cref{thm:maingeneral} with \cref{lem:pointed GHP conv} and Skorohod representation theorem we deduce the following.

\begin{theorem}[Pointed convergence]\label{thm:maingeneralpointed}
Let $\{ G_n \}$ be a sequence of graphs satisfying \cref{assn:main}, let $\T_n$ be a sample of $\UST(G_n)$ and let $O_n$ be an arbitrary vertex of $G_n$.  Denote by $d_{\T_n}$ the graph distance on $\T_n$ and by $\mu_n$ the uniform probability measure on the vertices of $\T_n$. Then there exists a sequence $\{\beta_n\}$ satisfying $0<\inf_n \beta_n \leq \sup_n \beta_n < \infty$ such that  
    \begin{equation*}
        	\left(\T_n,\frac{d_{\T_n}}{\beta_n \sqrt{n}}, \mu_n, O_n \right) \overset{(d)}{\longrightarrow} \left( \T, d, \mu, O \right)
    \end{equation*}
where $(\T,d_{\T},\mu, O)$ is the {\rm CRT} equipped with its canonical mass measure $\mu$ and root $O$, and $\overset{(d)}{\longrightarrow}$ means convergence in distribution with respect to the  pointed {\rm GHP} distance defined in \cite[Section 2.2]{Croydon18Scaling}.
\end{theorem}

\subsubsection{Diameter distribution}
	The {\bf diameter} of a metric space $(X,d)$ is $\sup_{x,y\in X}d(x,y)$ and denoted by $\diam(X)$. When $X$ is a tree, it is just the length of the longest path. The study of the diameter of random trees has an interesting history. Szekeres \cite{szekeres} proved in 1982 that the diameter $D_n$ of a uniformly drawn labeled tree on $n$ vertices normalized by $n^{-1/2}$ converges in distribution to a random variable $D$ with the rather unpleasant density
	\begin{equation}\label{eq:yuckydist} f_D(y) = {\sqrt{2\pi}\over 3} \sum_{n\geq 1} e^{-b_{n,y}} \Big ( {64 \over y^2} ( 4b_{n,y}^4 - 36b_{n,y}^3 +75 b_{n,y}^2 -30b_{n,y}) + {16 \over y^2} (2b_{n,y}^3 - 5b_{n,y}^2) \Big ) \, ,\end{equation}
	where $b_{n,y} = 8(\pi n/y)^2$ and $y\in(0,\infty)$. Aldous \cite{AldousCRTI, AldousCRTII} showed that this tree, viewed as a random metric space, converges to the CRT and deduced that $D$ is distributed as 
	\begin{equation}\label{eq:prettydist} 2 \cdot \sup_{0 \leq t_1 < t_2 \leq 1} \big( e_{t_1} + e_{t_2} - 2 \inf_{t_1 \leq t \leq t_2} e_t \big )\, ,\end{equation}
	where $\{e_t\}_{t \in [0,1]}$ is standard Brownian excursion. Curiously enough, up until 2015 the only known way to show that \eqref{eq:prettydist} has density \eqref{eq:yuckydist} was to go via random trees and combine the Aldous and Szekeres results. Wang \cite{wang2015heightdiam}, prompted by a question of Aldous, gave a direct proof of this fact in 2015.
	
	A uniformly drawn labeled tree on $n$ vertices is just $\UST(K_n)$ where $K_n$ is the complete graph on $n$ vertices. Applying \cref{thm:maingeneral} we are able to extend Szekeres' 1983 result to USTs of any sequence of graphs satisfying \cref{assn:main}. 
	
	\begin{corollary} \label{cor:diam} Let $\{G_n\}$ be a sequence of graphs satisfying \cref{assn:main}, let $\T_n$ be a sample of $\UST(G_n)$ and let $\{\beta_n\}$ be the sequence guaranteed to exist by \cref{thm:maingeneral}. Then 
		$$ {\diam(\T_n) \over \beta_n n^{1/2}} \overset{(d)}{\longrightarrow} D \, ,$$
		where $D$ is the diameter of the {\rm CRT}, i.e., a random variable defined by either \eqref{eq:yuckydist} or \eqref{eq:prettydist}.
	\end{corollary}
	\begin{proof} Let $D_n = \diam(\T_n)$ and let $g:[0,\infty)\to\RR$ be bounded and continuous. The function $h:\XX_c\to \RR$ defined by $h((X,d,\mu))=\diam(X)$ is continuous with respect to the GHP topology; indeed, for any two metric spaces $X_1$ and $X_2$ we have $|\diam(X_1) - \diam(X_2)|\leq 2\dGHP(X_1,X_2)$. Thus the composition $g \circ h : \XX_c \to \RR$ is bounded and continuous. By \cref{thm:maingeneral} we conclude $\E \big [ g\circ h ((\T_n, {d_{\T} \over \beta_n \sqrt{n}},\mu_n) \big ] \to \E [g\circ h ((\T,d,\mu))]$ where $(\T,d,\mu)$ is the CRT. Therefore, $\E[g(D_n)] \to \E[g(D)]$ as required. 
	\end{proof}

	\subsubsection{Height distribution}
	Given a  rooted tree $(T,v)$, the {\bf height} of $(T,v)$ is $\sup_{x\in T} d(v,x)$, i.e. the length of the longest simple path in $T$ starting from $v$, and denoted by $\height(T,v)$. The study of the height of random trees predates the study of the diameter. In 1967, R\'enyi and Szekeres \cite{renyi1967height} found the limiting distribution of the height of a uniformly drawn labeled rooted tree on $n$ vertices normalized by $n^{-1/2}$; we omit the precise formula this time (it is also unpleasant). Aldous \cite{AldousCRTI, AldousCRTII} realized that the limiting distribution is that of the maximum of the Brownian excursion. 
	
	The following corollary is an immediate consequence of \cref{thm:maingeneralpointed}. The proof goes along the same lines as the proof of \cref{cor:diam}; we omit the details.
	
	\begin{corollary} \label{cor:height} Let $\{G_n\}$ be a sequence of graphs satisfying \cref{assn:main}, let $\T_n$ be a sample of $\UST(G_n)$ and let $\beta_n$ be the sequence guaranteed to exist by \cref{thm:maingeneral}. Let $v_n$ be an arbitrary vertex of $G_n$. Then 
		$$ {\height(\T_n,v_n) \over \beta_n n^{1/2}} \overset{(d)}{\longrightarrow} 2 \sup_{t\in [0,1]} e_t  \, ,$$
		where $\{e_t\}_{t\in [0,1]}$ is standard Brownian excursion.
	\end{corollary}

\subsubsection{SRW on the UST converges to BM on the CRT}
A particularly nice application of \cref{thm:maingeneral} together with \cite[Theorem 1.2]{Croydon18Scaling} allows us to deduce that the simple random walk (SRW) on $\UST(G_n)$ rescales to Brownian motion on the CRT. The latter object was first defined by Aldous in \cite[Section 5.2]{AldousCRTII} and formally constructed by Krebs \cite{Krebs}.
In what follows, we let $P^{(O_n)}_n( \cdot )$ denote the (random) law of a discrete-time SRW on $\UST (G_n)$, started from $O_n$, and let $P^{(O)}( \cdot )$ denote the law of Brownian motion on the CRT as constructed by Krebs, started from $O$.

\begin{theorem}\label{thm:SRW}
Let $\{ G_n \}$ be a sequence of graphs satisfying \cref{assn:main}, let $\T_n$ be a sample of $\UST(G_n)$, and let $(X_n(m))_{m \geq 0}$ be a simple random walk on $\T_n$. 
Then there exists a probability space $\Omega$ on which the convergence of \cref{thm:maingeneralpointed} holds almost surely, and furthermore, on this probability space, for almost every $\omega \in \Omega$ the spaces $((\T_n, d_n, \mu_n, O_n))_{n \geq 1}$ and $(\T, d, \mu, O)$ can be embedded into a common metric space $(X', d')(\omega)$ so that 
\begin{equation}\label{eqn:SRW law conv}
P^{(O_n)}_n \left( \left(\frac{1}{\beta_n \sqrt{n}} X_n(2\beta_n n^{\frac{3}{2}}t)\right)_{t \geq 0} \in \cdot \right) \to P^{(O)}\left( (B_t)_{t \geq 0} \in \cdot \right)
\end{equation}
weakly as probability measures on the space $D(\RR^{\geq 0}, X'(\omega))$ of c\`adl\`ag functions equipped with the uniform topology.
\end{theorem}
\begin{proof}

The existence of such a probability space $\Omega$ follows from the Skorohod representation theorem since the space of pointed compact mm-spaces endowed with a finite measure is separable by \cite[Theorem 6 and Proposition 8]{MiermontTessellations}. 
The theorem now follows from \cite[Theorem 1.2]{Croydon18Scaling} and two additional observations.

Firstly, if $\nu_n (x) = \deg x$, then $\dGHP \left( (\T_n, \frac{d_{\T_n}}{\beta_n \sqrt{n}}, \mu_n), (\T_n, \frac{d_{\T_n}}{\beta_n \sqrt{n}}, \frac{1}{2n}\nu_n) \right) \leq \frac{1}{\beta_n \sqrt{n}}$, so that 
    \begin{equation*}
        	\left(\T_n,\frac{d_{\T_n}}{\beta_n \sqrt{n}}, \frac{1}{2n}\nu_n \right) \overset{(d)}{\longrightarrow} \left( \T, d, \mu \right)
	    \end{equation*}
with respect to the $\GHP$ distance as a consequence of \cref{thm:maingeneral} and the triangle inequality. It therefore follows from \cite[Theorem 1.2]{Croydon18Scaling} that if 
$(Y_n(t))_{t \geq 0}$ is a continuous time SRW on $G_n$ with an $\textsf{exp}(1)$ holding time at each vertex, then
\begin{equation}\label{eqn:Yn conv to BM}
\left(\frac{1}{\beta_n \sqrt{n}} Y_n(2\beta_n n^{\frac{3}{2}}t)\right)_{t \geq 0} \overset{(d)}{\longrightarrow} (B_t)_{t \geq 0}
\end{equation}
as $n \to \infty$, almost surely on $\Omega$.
This result then transfers to the SRW sequence $(X_n(\cdot))_{n \geq 1}$ in place of $(Y_n(\cdot))_{n \geq 1}$ by standard arguments using the strong law of large numbers and continuity of the limit process. We refer to \cite[Section 4.2]{Archer2019} for an example of such an argument.
\end{proof}

\subsection{Organization}
We begin with some preliminaries in \cref{sec:prelim} where we introduce the standard definitions of loop-erased random walk, mixing time and capacity which are central to the proof. We also record some stochastic domination properties of USTs, and prove there a general result regarding negative correlations of certain expected volumes in the $\UST$ (see \cref{cl: expected stoch dom}).

Next in \cref{sec:high-level} we present the main argument of the proof, while delegating two useful estimates, \cref{thm:gammaisnice} and \cref{thm:capacityofprefix}, to \cref{sctn:cap bounds}, and a third useful estimate, \cref{lem:interval test tail Bj G'}, to \cref{sctn:conditional tail bound}. In \cref{sec:abstract} we present a necessary though rather straightforward abstract argument combining the result of \cref{sec:high-level} with the results of \cite{PeresRevelleUSTCRT} to yield \cref{thm:maingeneral}. Lastly, in \cref{sec:the_end} we present some concluding remarks and open questions.

\subsection{Acknowledgments}
We thank Christina Goldschmidt for many useful discussions. This research is supported by ERC starting grant 676970 RANDGEOM, consolidator grant 101001124 UniversalMap, and by ISF grant 1294/19.

\section{Preliminaries}\label{sec:prelim}
In this section we provide an overview of the tools used to prove \cref{thm:maingeneral}. Throughout the section, we assume that $G=(V,E)$ is a finite connected graph with $n$ vertices. We will use the following conventions:

\begin{itemize}
	\item For an integer $m\geq 1$ we write $[m]=\{1,\ldots, m\}$.
	\item For two positive sequences $t(n), r(n)$ we write $t \sim r$ when $t(n)/r(n) \to 1$.
	\item For two positive sequences $t(n), r(n)$ we write $t \gg r$ when $t(n)/r(n) \to \infty$.
	\item We omit floor and ceiling signs when they are necessary.
	\item Through the rest of this paper, the \emph{random walk} on a graph equipped with positive edge weights is the random walk that stays put with with probability $1/2$ and otherwise jumps to a random neighbor with probability proportional to the weight of the corresponding edge. If no edge weights are specified, then they are all unit weights.
\end{itemize}
	
\subsection{Loop-erased random walk and Wilson's algorithm}\label{sctn:Wilson}
Wilson's algorithm \cite{Wil96}, which we now describe, is a widely used algorithm for sampling $\UST$s.
A {\bf walk} $X=(X_0, \ldots X_L)$ of length $L\in\mathbb{N}$ is a sequence of vertices where $(X_i, X_{i+1})\in E(G)$ for every $0 \leq i \leq L-1$. For an interval $J=[a,b]\subset[0,L]$ where $a,b$ are integers, we write $X[J]$ for $\{X_i\}_{i=a}^{b}$. Given a walk, we define its {\bf loop erasure} $Y = \LE(X) = \LE(X[0,L])$ inductively as follows. We set $Y_0 = X_0$ and let $\lambda_0 = 0$. Then, for every $i\geq 1$, we set $\lambda_i = 1+\max\{t \mid X_t = Y_{\lambda_{i-1}}\}$ and if $\lambda_i \leq L$ we set $Y_i = X_{\lambda_i}$. We halt this process once we have $\lambda_i > L$. The times $\langle \lambda_k(X)\rangle_{k=0}^{|\LE(X)| - 1}$ are the {\bf times contributing to the loop-erasure} of the walk $X$. When $X$ is a random walk starting at some vertex $v\in G$ and terminated when hitting another vertex $u$ ($L$ is now random), we say that $\LE(X)$ is the {\bf loop erased random walk} ($\LERW$) from $v$ to $u$.  
	
To sample a $\UST$ of a finite connected graph $G$ we begin by fixing an ordering of the vertices of $V=(v_1,\ldots, v_n)$. At the first step, let $T_1$ be the tree containing $v_1$ and no edges. At each step $i>1$, sample a $\LERW$ from $v_i$ to $T_{i-1}$ and set $T_i$ to be the union of $T_{i-1}$ and the $\LERW$ that has just been sampled. We terminate this algorithm with $T_n$. Wilson \cite{Wil96} proved that $T_n$ is distributed as $\UST(G)$. An immediate consequence is that the path between any two vertices in $\UST (G)$ is distributed as a $\LERW$ between those two vertices. This was first shown by Pemantle \cite{Pem91}.
	
To understand the lengths of loops erased in $\LERW$ we will need the notion of the \textbf{bubble sum}. Let $G$ be a graph and let $W$ be a non empty subset of vertices of $G$. For every two vertices $u,w\in V(G)$, define
	\begin{equation*}
	\pvec{W}^t(u,w) = \pr_u(X_t = w, X[0,t]\cap W = \emptyset) \, ,
	\end{equation*}
	where $X$ is a random walk on $G$. We define the $W$-bubble sum by
	\begin{equation*}
	\mathcal{B}_W(G) := \sum_{t=0}^{\infty}(t+1)\sup_{v\in V} \pvec{W}^t(v,v).
	\end{equation*}
	 Note that since the random walk on $G$ is an irreducible Markov chain on a finite state space, we have that  $\pr(X[0,t] \cap W = \emptyset)$ decays exponentially in $t$ and hence this sum is always finite. Another bubble-sum we will consider is when the random walk is killed at a geometric time (rather than when hitting a set $W$). Let $T_\zeta$ be an independent geometric random variable with mean $\zeta>1$. We define
	\begin{equation*}
		\pvec{\zeta}^t (u,w) = \pr_u(X_t = w, T_\zeta>t), \quad \mathcal{B}_\zeta(G) := \sum_{t=0}^{\infty}(t+1)\sup_{v\in V} \pvec{\zeta}^t(v,v).
	\end{equation*}
	
\begin{definition}\label{def:bubbleterm} We say that a random walk $X$ on a finite connected graph $G$ starting from an arbitrary vertex is {\bf bubble-terminated} with bubble-sum bounded by $\secp$ if it is killed upon hitting some set $W$ and $B_W(G) \leq \secp$, or alternatively, if it is killed at time $T_\zeta-1$ and $B_{\zeta}(G) \leq \secp$.\end{definition}

	Both bubble-sums allow us to bound the size of the loops erased in the loop-erasure process. As in \cite{Hutchcroft2020universality} and \cite[Claim 3.2]{MNS} we have the following.
	
	\begin{claim}\label{cl:indep:loops} Let $G$ be a finite connected graph and $X$ be a bubble-terminated random walk on $G$ with bubble-sum bounded by $\secp$. For any finite simple path $\gamma$ such that $\pr(\LE(X)=\gamma)>0$ of length $L$ we have that the random variables 
	$$ \left\{ \lambda_{i+1}(X) - \lambda_i(X) \right\}_{i=0}^{L-1} $$
	are independent conditionally on $\{\LE(X) = \gamma\}$ and furthermore
	\begin{equation*}
		\E[\lambda_{i+1}(X) - \lambda_{i}(X) | \LE(X)=\gamma] \leq \secp\, ,
	\end{equation*}
	for all $0\leq i\leq L-1$.
	\end{claim}
	\begin{proof} In the case that $X$ is killed upon hitting $W$, see \cite[Proof of Claim 3.2]{MNS}. (Note that the definitions of the times contributing to the loop erased random walk are a little different. More accurately, the $k$th time contributing to the loop erasure according to our definition equals the $(k+1)$th time contributing to the loop erasure minus one according to the definition in \cite{MNS}, thus the expected difference between two consecutive times has the same bound, and the proof is the same.)

	When $X$ is killed at $T_\zeta-1$ where $T_\zeta$ is an independent geometric random variable with mean $\zeta>1$, the proof can be deduced from the previous claim. Indeed, we add a new vertex $\rho$ to $G$ and edges $(\rho,u)$ for every $u\in G$ with weights on them so that the probability to visit $\rho$ from every vertex in a single step is equal to $1/\zeta$ for any $u\in G$. Call the resulting network $G^*$. A random walk on $G^*$ started from $v\in G$ and terminated when hitting $\rho$ has the same distribution as a random walk on $G$ with geometric killing time. 
	\end{proof}

	\subsection{Mixing times} 
	Recall the definition of the uniform mixing time above \cref{assn:main}. 
	It follows that for every $t\geq \tmix$ we have that
	\begin{equation}\label{eq:hitting:prob:near:station}
		\frac{1}{2n} \leq \pr_u(X_t = v) \leq \frac{2}{n} \, ,
	\end{equation}
	where $X_t$ is the random walk. Even though in this paper we mainly use the uniform mixing time as defined in \eqref{def:tmix} we also use a more classical version of distance between probability measures on finite sets. Recall that the {\bf total variation distance} between two probability measures on $\mu$ and $\nu$ on a finite set $X$ is defined by
	$$ d_{\textrm{TV}}(\mu,\nu) = \max_{A \subset X} |\mu(A)-\nu(A)| \, .$$
	It is a standard fact (see \cite[Section 4.5]{levin2017markovmixing}) that if $t \geq k\tmix$, then for any vertex $x$
	\begin{equation}\label{eq:tvDecay} d_{\textrm{TV}}(p_t(x,\cdot),\pi(\cdot)) \leq 2^{-k} \, .\end{equation}

\subsection{Capacity}
The \emph{capacity} of a set of vertices quantifies how difficult it is for a random walk to hit the set. It is a crucial notion when one wishes to analyze the behavior of Wilson's algorithm. 
Let $\{Y_i\}_{i\geq 0}$ be a random walk on $G$ and for $U \subset V(G)$, let $\tau_U = \inf \{i \geq 0: Y_i \in U\}$. Given $k \geq 0$ we define the $\mathbf{k}$\textbf{-capacity} of $U$ by $\Capk(U) = \prstart{\tau_U \leq k}{\pi}$. If $W \subset V(G)$ is another subset of vertices we define the \textbf{relative} $\mathbf{k}$\textbf{-capacity} $\Capk(W, U) = \prstart{\tau_W \leq k, \tau_W \leq \tau_U}{\pi}$. Note that the relative capacity is \emph{not} symmetric in $W,U$.

We will see later that the capacities of certain subsets determine the expected volumes of balls in $\UST (G)$. Here we collect some useful facts about the capacity. By the union bound, when $G$ is regular we always have the upper bound
	\begin{equation}\label{eqn:cap UB}
	\Capk (V) \leq k\pi (V)= \frac{k |V|}{n}.
	\end{equation}

The capacity is defined for the lazy simple random started at stationarity. When $k$ is significantly larger than the mixing time, the starting vertex does not make much difference as the following claim shows.
	\begin{claim}\label{cl:cap:mix}
		Let $G$ be a connected regular graph. Let $u\in V$, let $U\subseteq V$ be nonempty,
		let $r=r(n)\gg\log(n)\cdot\tmix(G)$ and assume that $t=t(n)$ is a sequence so that $t(n) \sim r(n)$. Then, for large enough $n$,
		\begin{equation*}
		\pr_u(\tau_U<t) \ge \frac{1}{3}\Capa_r(U).
		\end{equation*}
	\end{claim}
	\begin{proof} See \cite[Claim 1.4]{MNS}.
	\end{proof}

We will also use the following lemma.
\begin{lemma}\label{lem:finddisjoint} Let $G$ be a connected regular graph. Let $W\subset U$ be subsets of vertices and $k,s,m\geq 0$. Assume that
$$ \Capk (W, U\setminus W) \geq s \, .$$
Then we can find at least $L= \lfloor s/(m+k/n) \rfloor$ disjoint subsets $A_1,\ldots, A_L$ of $W$ such that 
$$   m \leq  \Capk (A_j, U \setminus A_j) \leq m + {k \over n}  \, ,$$
for all $j=1,\ldots, L$.
\end{lemma}
\begin{proof}  We first observe that if $A$ is a subset of $W$ and $v \in W \setminus A$ then by \eqref{eqn:cap UB} we have that
$$  \Capk ( A \cup \{v\}, U \setminus A\cup\{v\}) \leq \Capk (A , U \setminus A ) + {k \over n} .$$
Secondly, we observe that if $A_1,\ldots, A_{L'}$ are disjoint sets so that $\cup_{j=1}^{L'} A_j = W$, then
$$ \Capk (W, U\setminus W) = \sum_{j=1}^{L'} \Capk (A_j, U \setminus A_j) \, .$$
With these two observations in place, we now perform an iterative construction of the subsets. We add vertices from $W$ to $A_1$ until the first time that $\Capk(A_1, U\setminus A_1) \geq m$. By the first observation we have that $\Capk(A_1, U\setminus A_1) \leq m + k/n$. Then we add vertices from $W\setminus A_1$ to $A_2$ until the first time that $\Capk(A_2, U\setminus A_2) \geq m$ and so forth. By the second observation we deduce that we can continue this way until at least $L = \left \lfloor \frac{s}{m + \frac{k}{n}} \right \rfloor$, concluding the proof.
\end{proof}

	In order to obtain useful lower bounds on the capacity, we state a well-known relationship between the capacity of a set $A$ and the Green kernel summed over $A$. Given a set $A \subset G$ and $k \in \NN$ we define 
	\begin{equation}\label{def:Mk}
	M^{(k)}(A) = \sum_{x, y \in A} G^{(k)}(x,y),
	\end{equation}
	where $G^{(k)}(x,y) = \E_x \left[\sum_{i=0}^k \mathbbm{1}\{X_i = y\}\right]$. This is useful due to the following characterization of capacity.
	\begin{lemma}\label{lem: var def of k-cap}
Let $G$ be a connected regular graph. For all $A \subset G$,
		\begin{equation*}
		\Capa_k (A) \geq \frac{k|A|^2}{2n M^{(k)}(A)}.
		\end{equation*}
	\end{lemma}
	\begin{proof}
		The proof is the same as that of \cite[Theorem 2.2]{Benjamini1995MCap}, but instead considering a stationary starting point distributed according to $\pi$, noting that $G^{(k)}(\pi, x) = \frac{k}{n}$ for all $x \in G_n$ by transitivity, and specifically using the measure $\mu(x) = \frac{\mathbbm{1}\{x \in A\}}{|A|}$.
	\end{proof}

	The following bound on $\E\left[M^{(k)} (P)\right]$ where $P$ is a random walk path will be useful.
	 
	\begin{lemma}\label{lem:mk} Let $G$ be a connected regular graph with $n$ vertices. Let $m$ and $k$ be two positive integers 
	 and let $P$ be a random walk path of length $m$ started at $v \in V(G)$. Then
	$$ \E \left[ M^{(k)}(P) \right] \leq 2m \sum_{t=0}^{m+k} (t+1) p_t(v,v) \, .$$
	\end{lemma}
	\begin{proof}
	The proof goes by the same argument as in \cite[Lemma 5.6]{Hutchcroft2020universality}.
	\end{proof}

	Furthermore, in order to lower bound the relative capacity, we define the $\mathbf{k}$\textbf{-closeness} of two sets $U$ and $W$ by 
	\begin{equation}\label{def:close}\close_k (U,W) = \prstart{\tau_{U} < k, \tau_{W} < k}{\pi} \, .\end{equation} 
	It follows from \cite[Lemma 5.2]{PeresRevelleUSTCRT} together with \eqref{eqn:cap UB} that on any finite connected regular graph $G$, if $W = X[0,T]$ where $X$ is a random walk on $G$ started at stationarity, and $T$ is a stopping time, then for any set $U \subset G$,
	\begin{equation}\label{eqn:expected closeness}
	\E\left[ \close_k (U, W) \right] \leq \frac{4\E[T]k \Capk (U)}{n} \leq \frac{4k^2 |U| \E T}{n^2}.
	\end{equation}
	
	Lastly, recall the two bubble sums defined in \cref{sctn:Wilson}. One of the uses of the capacity is to bound such bubble sums.

\begin{claim}\label{claim: cap to bub UB} Let $\{ G_n \}$ be a sequence of graphs satisfying \cref{assn:main} and let $W\subset G_n$ be a set of vertices such that $\Capp_{\sqn}(W) \geq c$, then 
		\[
		B_{W}(G) \leq \theta + \frac{4}{c^2}.
		\]
\end{claim}
	
	\begin{proof}
	This follows by exactly the same proof as in \cite[Claim 3.14]{MNS}. 
	\end{proof}

\begin{claim} \label{claim:bubblezeta} Let $\{ G_n \}$ be a sequence of graphs satisfying \cref{assn:main} and let $\zeta > 0$ be given. Then
\begin{align*}
\mathcal{B}_{\zeta^{-1} n^{1/2}}(G_n) \leq \theta + 2\zeta^{-2}.
\end{align*}
\end{claim}
\begin{proof}
Take any $v \in G_n$. Then, similarly to \cite[Claim 3.14]{MNS}, since $\sum_{t=0}^{\infty}(t+1) \left( 1 - x \right)^t = x^{-2}$ and using \eqref{eq:hitting:prob:near:station},
\begin{align*}
\mathcal{B}_{\zeta^{-1} n^{1/2}}(G_n) = \sum_{t=0}^{\infty}(t+1) \pvec{}^t(v,v) \left( 1 - \zeta n^{-1/2} \right)^t &\leq \theta + \frac{2}{n}\sum_{t=\sqrt{n}}^{\infty}(t+1) \left( 1 - \zeta n^{-1/2} \right)^t \leq \theta + 2\zeta^{-2}.\qedhere
\end{align*}
\end{proof}

\subsection{Stochastic domination properties}\label{sctn:stoch dom lemmas}
	The $\UST$ enjoys the negative correlation property, i.e., the probability that an edge $e$ is such that $e \in \UST(G)$ conditioned on $f\in \UST(G)$ for some other edge $f$ is no more than the unconditional probability. Moreover, Feder and Mihail showed that for every increasing event $\mathcal{A}$ that ignores $f$, the probability of $\mathcal{A}$ given $f\in \UST(G)$ is no more than the unconditional probability. This led to the following result.
	
	\begin{lemma}\cite[Lemma 10.3]{LyonsPeres}.\label{lem:stoch dom subsets}
		Let $G$ be a connected subgraph of a finite connected graph $H$. Then, $\UST(G)$ stochastically dominates $\UST(H)\cap E(G)$. 		
	\end{lemma}
	
	
	The same proof leads to a slightly more generalized version.
	
	\begin{lemma}\label{lem:stoch dom add sun}
		Let $(G,w)$ be a weighted network and suppose that $(H,w')$ is a network such that $V(G) \subseteq V(H)$ and that for every edge $(v,u)$ with $w((v,u))  \neq 0$ we have $w((v,u)) = w'((v,u))$. Then, $\UST(G)$ stochastically dominates $\UST(H)\cap E(G)$.
	\end{lemma}

	Later in the paper, we will apply \cref{lem:stoch dom add sun} in the following context. To study $\UST (G)$ using Wilson's algorithm, it will sometimes be convenient to add an extra vertex to $G$ called the sun, and for every vertex $v \in G$ add an extra edge from $v$ to the sun. We give well-chosen weights to these new edges and call the new graph the \textbf{sunny graph}. \cref{lem:stoch dom add sun} tells us that the UST of the sunny graph, intersected with $E(G)$, is stochastically dominated by $\UST(G)$. This idea was previously used in \cite{Wil96} and \cite{PeresRevelleUSTCRT}.


	We will also make use of the following well-known lemma. Here $G/A$ denotes the graph obtained from $G$ by identifying all vertices in $A$ with a single vertex.
	
	\begin{lemma}\cite[Exercise 10.8]{LyonsPeres}.\label{lemma: stoch dom con}
		Let $(G,w)$ be a finite network. Let $A\subseteq B$ be two sets of vertices. Then, $\UST(G/A)$ stochastically dominates $\UST(G/B)$.
	\end{lemma}

	Lastly, let $W$ be a set of vertices, and let $A_1$ and $A_2$ be disjoint subsets of $W$. In what follows we consider $\UST(G/W)$. Given an integer $k$ and $j\in \{1,2\}$, let $I_j(k)$ denote the vertices of $G$ that are connected to $W$ in $\UST(G/W)$ by a path of length $k$ such that the last edge on the path to $W$ is an edge that one of its original endpoints belonged to $A_j$ (including $A_j$ itself). Also, let $X_j = X_j(k) = |I_j(k)|$.
	
	\begin{claim}\label{cl: expected stoch dom}
		Let $G$ be a finite connected graph, take any $k\geq 1$ and let $W, A_1, A_2$ be as above. Then, for $\UST(G/W)$ and for every $M>0$,
		\begin{equation*}
			\E\left[X_2 \middle\vert X_1 \leq M\right] \geq \E\left[X_2\right].
		\end{equation*}
	\end{claim}
	
	\begin{proof}
		We will first show that for every $v\in G$, the events $\{X_1 > M\}$ and $\{v\in I_2\}$ are negatively correlated. 
		Fix some $v\in G$ such that $v\in I_2$ has positive probability. Condition on $v\in I_2$ and on $\gamma_2$, the path from $v$ to $A_2$. The $\UST$ conditioned on $W$ and $\gamma_2$ has the distribution of $\UST(G/(W\cup\gamma_2))$. 
		Hence, by \cref{lemma: stoch dom con} we have that $\UST(G/(W\cup\gamma_2))$ is dominated by $\UST(G/W)$. 
		Therefore, by Strassen's theorem \cite{Strassen}, there exists a coupling of the two measures such that $\UST(G/(W\cup\gamma_2)) \subseteq \UST(G/W)$. This means that every vertex connected to $W$ through $A_1$ by a path of length at most $k$ in $\UST(G/(W\cup\gamma_2))$ will also be connected by the same path to $A_1$ in $\UST(G/W)$. Therefore, we have that 
		\begin{equation*}
			\pr\left(X_1 > M \mid v \in I_2(k), \gamma_2 \right) \leq \pr(X_1 > M).
		\end{equation*}
		Then by averaging over $\gamma_2$ and taking complements we obtain 
		\begin{equation*}
			\pr\left(X_1 \leq M \mid v \in I_2(k) \right) \geq \pr(X_1 \leq M).
		\end{equation*}
		Therefore, inverting using Bayes' rule, we have for every $v$ with $\pr(v\in I_2(k)) >0$ that
		\begin{equation*}
			\pr(v\in I_2(k) \mid X_1 \leq M) \geq \pr(v \in I_2(k)).
		\end{equation*}
		Summing over $v$ yields the result.
	\end{proof}
	
\section{The lower mass bound}\label{sec:high-level}

The starting point of the proof of \cref{thm:maingeneral} is the work of Peres and Revelle \cite{PeresRevelleUSTCRT}.
\begin{theorem}\cite[Theorem 1.2]{PeresRevelleUSTCRT}.\label{thm:pr04} Let $\{G_n\}$ be a sequence of graphs satisfying \cref{assn:main} and let $\T_n$ be $\UST(G_n)$. Denote by $d_{\T_n}$ the graph distance on $\T_n$ and by $(\T,d,\mu)$ the CRT. Then there exists a sequence $\{\beta_n\}$ satisfying $0<\inf_n \beta_n \leq \sup_n \beta_n < \infty$ such that the following holds. For fixed $k \geq 1$, if $\{x_1,\ldots, x_k\}$ are uniformly chosen independent vertices of $G_n$, then the distances 
$$ \frac{d_{\T_n}(x_i,x_j)}{\beta_n \sqrt{n}} $$
converge jointly in distribution to the ${k \choose 2}$ distances in $\T$ between $k$ i.i.d.~points drawn according to $\mu$. 
\end{theorem}

For the proof of \cref{thm:maingeneral} we take the same sequence $\beta_n$ guaranteed to exist by the theorem above. 
As we shall see in \cref{sec:abstract}, the convergence of \cref{thm:pr04} is equivalent to what is known as Gromov-weak convergence, which does not imply GHP convergence. In order to close this gap in this abstract theory, Athreya, L\"ohr and Winter \cite[Theorem 6.1]{AthreyaLohrWinterGap} introduced the \emph{lower mass bound} condition and proved that this condition together with Gromov-weak convergence is in fact equivalent to GHP convergence; we discuss this further in \cref{sec:abstract}. The main effort in this paper is proving that the lower mass bound holds under \cref{assn:main}; this is the content of the following theorem. 

\begin{theorem} \label{thm:lowermassboundUSTs} Let $\{ G_n \}$ be a sequence of graphs satisfying \cref{assn:main} and let $\T_n$ be $\UST(G_n)$. For a vertex $v\in\T_n$ and some $r\geq 0$ we write $B_{\T_n}(v,r) = \{u : d_{\T_n}(v,u)\leq r\}$ where $d_{\T_n}$ is the intrinsic graph distance metric on $\T_n$. Then for any $c>0$ and any $\delta>0$ there exists $\eps>0$ such that for all $n\geq 1$,
$$ \pr \big ( \exists v \in \T_n : |B_{\T_n}(v, c\sqrt{n})|\leq \eps n \big) \leq \delta.$$In other words, the random variables $\big\{\max_v \{n |B_{\T_n}(v, c\sqrt{n})|^{-1}\}\big\}_n$ are tight.
\end{theorem}

In the rest of this section we prove \cref{thm:lowermassboundUSTs}, delegating parts of the proof to \cref{sctn:cap bounds} and \cref{sctn:proof on conditional sunny graph}. For the rest of this section as well as \cref{sctn:cap bounds}, $\{ G_n \}$ is a sequence of graphs satisfying \cref{assn:main} and $\T_n$ is $\UST(G_n)$.

\subsection{Bootstrap argument}\label{sctn:bootstrap}
The main difficulty in \cref{thm:lowermassboundUSTs} is that it is global; that is, it requires a lower tail bound on the volumes of the balls around \emph{all} vertices simultaneously. Our approach is to prove a strong enough local version of this bound, that is, a bound for a single vertex, and use a bootstrap argument to obtain a weaker (yet sufficient) global bound. The idea is to use the observation that if there is one vertex $x \in \T_n$ such that $|B(x, r)|$ is small, then either $|B(x, \frac{r}{2})|$ is also small, or otherwise there are many vertices $v \in B(x, \frac{r}{2})$ such that $|B(v, \frac{r}{2})|$ is small. Provided that these two latter events are sufficiently less likely than the former, this allows us to define a sequence of events on the balls of dyadic radii $|B(x, \frac{r}{2^\ell})|$, each with strictly stronger tail decay than the previous one. We will iterate this 
observation enough times until the probability of the final event is $o\left(\frac{1}{n}\right)$, at which point we will apply the union bound and conclude the proof. 

Thus, our goal will be to iteratively improve the bounds on 
\begin{equation}\label{eq:mainbound} \pr \left(\left|B\left(x, \frac{c \sqrt{n}}{2^\ell}\right)\right| \leq {\eps \over 4^\ell} \left( \frac{c \sqrt{n}}{2^\ell}\right)^2\right) \, ,\end{equation}
where $x$ is a fixed vertex (our graphs are transitive so the choice of $x$ does not matter) and $\ell = 0,\ldots, N_n$ where $N_n$, the number of iterations, will be chosen suitably as we now explain.

Since we will use Wilson's algorithm to sample branches in $\UST(G_n)$, it will be important in our arguments in \cref{sctn:cap bounds} that the radius $\frac{c \sqrt{n}}{2^l}$ we consider at each step is significantly longer than the mixing time of a random walk on $G_n$. Therefore, we require that $\frac{c \sqrt{n}}{2^{N_n}} \gg n^{\frac{1}{2} - \alpha}$ (recall the constant $\alpha$ from \cref{assn:main}), so the number of iterations $N_n$ can be at most of order $\log n$. We will see in the proof of \cref{thm:lowermassboundUSTs} that for this bootstrap argument to work with only $\log n$ steps, it will be convenient to obtain bounds on \eqref{eq:mainbound} that are sub-polynomial in $\eps$.

A natural strategy to bound the probability in \eqref{eq:mainbound} is to first sample a single branch joining $x$ to a pre-defined root of $\UST(G_n)$, consider the volumes of balls in subtrees attached to this branch close to $x$, and show that the sum of these volumes is very unlikely to be too small. This strategy almost gives sufficiently strong tail decay, but there is one step at which the tail decay is not sub-polynomial. This problem arises in the first step since there is a probability of order $\eps$ that the path joining $x$ to a root vertex is of length less than $\sqrt{\eps n}$.
	
This is not a fundamental problem since if this path is short, then it means we just picked a short branch when longer branches to different roots were available. However, it is not convenient to condition on picking a long branch to a well-chosen root since this conditioning reveals too much information about $\UST (G_n)$, which makes it difficult to control other properties of the branch, primarily its capacity and the capacity of its subsets. It is also inconvenient (though probably possible) to continue choosing a few more branches until we reach a certain length.

The simplest way we found to circumvent this issue is to first sample a branch $\Gamma_n$ between two uniformly chosen vertices of $G_n$ and perform the bootstrap argument discussed above conditioned on $\Gamma_n$ and the event that it is a ``nice'' path, a property we will define later that will include, amongst others, the event that $\Gamma_n$ is not too short. Then, using Wilson's algorithm we may sample other branches of $\UST(G_n)$ by considering loop-erased random walks terminated at $\Gamma_n$; thus $\Gamma_n$ can be thought of as the backbone of $\UST (G_n)$, and provided $\Gamma_n$ is sufficiently long we can sample the branch from $x$ to $\Gamma_n$ and consider its extension into $\Gamma_n$ to make it longer if necessary. With this modified definition of a branch, it is then possible to prove a conditional sub-polynomial tail bound in $\eps$ for \eqref{eq:mainbound}, and then to prove \cref{thm:lowermassboundUSTs} by decomposing according to whether $\Gamma_n$ is ``nice'' or not.

Throughout the rest of this paper, and in accordance with \cref{thm:lowermassboundUSTs}, we fix $c>0$ to be a small enough parameter and $\eps>0$ which can also be chosen to be small enough depending on $c$ and set
\begin{equation}\label{eq:parameters1} N_n = \frac{\alpha}{10} \log_2 n \qquad r = c\sqrt{n},\end{equation}
and for any scale $\ell \in \{0,\ldots, N_n\}$ 
\begin{equation}\label{eq:parameters2} r_\ell = \frac{r}{2^\ell} \qquad \epsilon_\ell = \frac{\epsilon}{4^\ell} \qquad k_\ell = \eps_\ell^{1/2} r_\ell \, .\end{equation}

\begin{theorem}\label{thm:gammaisnice} 
Let $\{ G_n \}$ be a sequence of graphs satisfying \cref{assn:main}, let $\T_n$ be $\UST(G_n)$ and denote by $\Gamma_n$ the unique path between two independent uniformly chosen vertices. Then for any $\delta>0$ there exist $c', \eps'>0$ such that for all $c\in(0,c')$ and all $\eps\in(0,\eps')$ there exists $N = N(\delta, c, \eps)$ such that for any $n\geq N$ we have that,	with probability at least $1-\delta$,
			\begin{enumerate}
			\item[$\mathrm{(I)}$] $\Capp_{\sqrt{n}} (\Gamma_n) \geq 2c$,
			\item[$\mathrm{(II)}$] For any scale $\ell \in \{0,\ldots, N_n\}$ and subsegment $I\subseteq \Gamma_n$ with $|I| = r_\ell/3$ we have that $$\Capkl(I, \Gamma_n \setminus I) \geq \frac{\eps_{\ell}^{1/6} k_\ell r_\ell}{n} \, ,$$
			\item[$\mathrm{(III)}$] $|\Gamma_n| \leq \epsilon^{-\frac{1}{10}}\sqrt{n}$.
		\end{enumerate}
\end{theorem}

\begin{definition}\label{def:enc} For the rest of this paper, given $c$ and $\eps$ as above we denote by $\Enc$ the intersection of the events in {\rm (I), (II), (III)} of the above theorem.
\end{definition}

\begin{remark} 
The reader may notice that although $c$ is fixed in \cref{thm:lowermassboundUSTs} and \eqref{eq:parameters1}, it is now treated as a variable parameter in \cref{thm:gammaisnice}. This is intentional since we need $|\Gamma_n| \geq c\sqrt{n}$ to overcome the problem of branch length mentioned above, and we cannot ensure this with high probability when $c$ is fixed; only when it is small. To prove \cref{thm:lowermassboundUSTs}, we start with a fixed $c$, but our first step is to reduce it if necessary so that the statement of \cref{thm:gammaisnice} holds as well. We will then prove the theorem with this smaller value of $c$. This poses no problem since the assertion of \cref{thm:lowermassboundUSTs} with smaller $c$ is stronger; this is also discussed in the proof of \cref{thm:lowermassboundUSTs}.
\end{remark}

Next we assume that $\Enc$ holds for some positive $c$ and $\eps$, and let $x$ be a vertex of $G_n$. Let $\Gamma_x$ denote the loop-erasure of the random walk path starting $x$ and stopped when it hits $\Gamma_n$ (if $x\in \Gamma_n$ then $\Gamma_x$ is empty). For an integer $s\in(0,c\sqrt{n})$ we denote by $\Gamma_x^s$ the prefix of $\Gamma_x$ of length $s$ as long as $|\Gamma_x|\geq s$; otherwise, i.e. if $|\Gamma_x|<s$, we denote by $\Gamma_x^s$ the prefix of the path in $\UST(G_n)$ from $x$ to one of the two endpoints of $\Gamma_n$ such that this path has length at least $s$. This is possible since by part (I) of \cref{thm:gammaisnice} and \eqref{eqn:cap UB}, if $\Enc$ holds, then $|\Gamma_n| \geq 2c\sqrt{n}$. If the two endpoints of $\Gamma_n$ can be used, we choose one in some arbitrary predefined manner. In \cref{sctn:capacity tail decay} we will prove the following. 
	
	\begin{theorem}\label{thm:capacityofprefix} Let $\{ G_n \}$ be a sequence of graphs satisfying \cref{assn:main} and let $\T_n$ be $\UST(G_n)$. Denote by $\Gamma_n$ the unique path between two independent uniformly chosen vertices and for a vertex $x \in G_n$ and $s>0$ let $\Gamma^s_x$ be as described above. Then for any $c >0$ there exist $\eps'>0$ and a constant $a > 0$ such that for every $\eps\in(0,\eps')$ there exists $N=N(c,\eps)$ such that for any $n\geq N$ and any $\ell \in \{0,\ldots, N_n\}$ we have
		\[
		\pr \left(\Capkl\left(\Gamma_x^{5r_\ell/6}, (\Gamma_n \cup \Gamma_x) \setminus \Gamma_x^{5r_\ell/6}\right) \leq \frac{\eps_\ell^{\frac{1}{6}} k_\ell r_\ell }{n} \ \mathrm{and} \ \Enc \right) \leq e^{-a (\log \eps_\ell^{-1})^2}.
		\]
	\end{theorem}

Given these two estimates we are now ready to proceed with the proof of \cref{thm:lowermassboundUSTs}.
Our strategy is as follows. On the event
	\begin{equation}\label{eq: good paths event}\left\{\Capkl\left(\Gamma_x^{5r_\ell/6}, (\Gamma_n \cup \Gamma_x) \setminus \Gamma_x^{5r_\ell/6}\right) \geq \frac{\eps_\ell^{\frac{1}{6}} k_\ell r_\ell }{n}\right\},
	\end{equation} we can condition on $\Gamma_n$ and $\Gamma_x$ and then apply \cref{lem:finddisjoint} with $m = \frac{2^{13}e\dumr \dumk \dumep^{\frac{1}{2}}}{4n}$ and $s=\frac{\dumr \dumk \dumep^{\frac{1}{6}}}{n}$ to obtain $L = \frac{\dumep^{-1/3}}{2^{12}e}$ disjoint subintervals $A_1, \ldots, A_L$ of $\Gamma_x^{5r_\ell/6}$ such that
	\[
	\Capkl (A_j, (\Gamma_n \cup \Gamma_x) \setminus A_j)) \geq \frac{2^{13}e\dumr\dumk\dumep^{\frac{1}{2}}}{4n} = \frac{2^{11}e\dumep\dumr^2}{n}
	\]
	for all $j=1,\ldots, L$. Moreover, since the cardinality of $\cup_{j=1}^L A_j$ is at most $\frac{5\dumr}{6}$, the number of $j$'s such that $|A_j| \geq 2^{13}e\dumep^{1 \over 3} \dumr$ is at most $\frac{5}{6\cdot 2^{13}e} \dumep^{-\frac13}$. Hence the number of $j$'s for which $|A_j| \leq (2^{13}e)\dumep^{1 \over 3} \dumr$ is at least $(2^{13}e)^{-1}\dumep^{-\frac13}$; we relabel the sets so that $A_i$ for $i=1,\ldots, (2^{13}e)^{-1}\dumep^{-\frac13}$ have this upper bound on their size and forget about the other sets.


	
	Our aim will be to test each of the intervals $\{A_i\}$ in turn to see if the trees hanging on $A_i$ contribute at least $\dumr^2 \dumep$ to $B(x, \dumr)$. We will test these intervals \textit{conditionally} on $\Gamma_x \cup \Gamma_n$ and on the outcome of the previous tests. Here we encounter a significant difficulty since the failure of some past tests introduces a complicated conditioning which we cannot access directly by contracting some edges. 

	
	To overcome this difficulty we proceed as follows. Conditioned on $\Gamma_n \cup \Gamma_x \subset \UST(G_n)$, we contract $\Gamma_n \cup \Gamma_x$ to a single vertex (still remembering the original edge-set) to form the graph $G_n / (\Gamma_n \cup \Gamma_x)$. By the UST spatial Markov property \cite[Proposition 4.2]{benjamini2001}, we have that $\UST(G_n)$ is distributed as the union of $\Gamma_n \cup \Gamma_x$ and the $\UST$ of this new graph. 
	Before proceeding, we then add a new vertex called {\bf the sun}, denoted by $\sunk$, to the graph $G_n / (\Gamma_n \cup \Gamma_x)$, and add an edge from every vertex to the sun with weight chosen so that a lazy random walk on $G_n \cup \{\sunk\} / (\Gamma_n \cup \Gamma_x)$ will always jump to the sun at the next step with probability $\frac{1}{\dumk}$. Then, we identify the sun with $\Gamma_n \cup \Gamma_x$, remembering the edges emanating from the sun. This ensures that when we run Wilson's algorithm on the remaining graph, rooted at the contracted vertex, random walks will always be killed when they hit the sun, so typically they only run for time of order $\dumk$.
	
	On the graph $G_n/(\{\sunk\} \cup \Gamma_n \cup \Gamma_x)$ we will often say \textbf{``hit $\mathbf{A}$"} when $A$ is a subset of $\{\sunk\} \cup \Gamma_n \cup \Gamma_x$. The meaning of hitting $A$ in the graph $G_n/(\{\sunk\} \cup \Gamma_n \cup \Gamma_x)$ is to hit $\{\sunk\} \cup \Gamma_n \cup \Gamma_x$ by traversing an edge whose original endpoint belonged in $A$. In some cases it will be convenient to start a random walk at a uniform vertex $U$ in the original graph $G_n$, and project the start point onto $G_n/(\{\sunk\} \cup \Gamma_n \cup \Gamma_x)$; in this case ``hit $A$" also includes the event $U \in A$.
%
	
	By \cref{lem:stoch dom add sun}, conditionally on $\Gamma_n \cup \Gamma_x$, we have that $\UST (G_n / (\Gamma_n \cup \Gamma_x))$ stochastically dominates $\UST(G_n / (\{\sunk\} \cup (\Gamma_n \cup \Gamma_x)) \cap E(G_n / (\Gamma_n \cup \Gamma_x))$. Therefore, we can couple the two $\UST$s together such that every edge $e$ not adjacent to the sun in $\UST(G_n / (\{\sunk\} \cup (\Gamma_n \cup \Gamma_x))$ also appears in $\UST(G_n / (\Gamma_n \cup \Gamma_x))$. When we expand $\{\sunk\} \cup\Gamma_n \cup \Gamma_x$ in $\UST(G_n / (\{\sunk\} \cup (\Gamma_n \cup \Gamma_x))$ and then remove $\sunk$ and its incident edges, we obtain several connected components, one of which contains $x$. By stochastic domination, the component containing $x$ is a subset of $\UST(G_n)$. Therefore, let $\Bsun(x, \dumr)$ denote the set of vertices connected to $x$ by a path of length at most $\dumr$ that does not intersect the sun after expanding $\{\sunk\} \cup\Gamma_n \cup \Gamma_x$ in the sunny graph. By stochastic domination, if we can prove a lower tail bound for $\Bsun(x, \dumr)$ on the sunny graph, it automatically transfers to a lower tail bound for $B_{\T_n}(x, \dumr)$ on the original graph.
	
	Recall that, given $\Gamma_n \cup \Gamma_x$, each of the $A_i$'s defined above is a subset of $\Gamma_n \cup \Gamma_x$. 
	When working on the graph $G_n/(\{\sunk\} \cup \Gamma_n \cup \Gamma_x)$, we let $I_i(\dumk)$ be the set of vertices connected to the contracted vertex in $\UST(G_n/(\{\sunk\} \cup \Gamma_n \cup \Gamma_x)$ by a path of length at most $\dumk$, such that the last edge on this path has an endpoint in $A_i$. Note that this is equivalent to being connected to $A_i$ by a path of length at most $\dumk$ not touching $\Gamma_n\cup \Gamma_x$ after expanding the path and separating $\sunk$ to obtain a subset of $\UST (G_n)$. We also include $A_i$ in $I_i(\dumk)$ and set $X_i = X_i(\dumk) = |I_i(\dumk)|$.
	
	Let $\Bj = \{ \sum_{i=1}^j X_i \leq 16\dumep \dumr^2\}$ and (for notational convenience) interpret $\Bz$ as an almost sure event. In \cref{sctn:conditional tail bound} we will prove the following lemma. 
	
	\begin{lemma}\label{lem:interval test tail Bj G'}
		Conditionally on $\Gamma_x \cup \Gamma_n$, let $\Bj$ be as defined above on the graph $G_n/(\{\sunk\} \cup \Gamma_n \cup \Gamma_x)$. Then for each $j\leq (2^{13}e)^{-1}\dumep^{-\frac13}$ 
		\[
		\pr \left( \Bj \middle\vert \Bjj,  (\Gamma_n \cup \Gamma_x), \Capkl (\Gamma_x^{5\dumr/6}  ,  \Gamma_n \cup \Gamma_x \setminus \Gamma_x^{5\dumr/6}) \geq \frac{  \dumr \dumk \dumep^{\frac{1}{6}}}{n} \right) \leq 1 - \frac{1}{160e}\dumep^{1/6}.
		\]
	\end{lemma}
	
	This has the following immediate corollary.
	
	\begin{corollary}\label{cor:volLBfromCap} 
		Let $\{ G_n \}$ be a sequence of graphs satisfying \cref{assn:main} and let $\T_n, \Gamma_n$ and $\Gamma_x$ be as in the previous theorem. Then for any $c>0$, any $\eps > 0$, all $n$ large enough and any $\ell \in \{0,\ldots, N_n\}$, we have
		\[
		\pr \left(  |B_{\T_n}(x,r_\ell)| \leq 16 r_\ell^2 \eps_\ell  \,\, ,  \,\, \Capkl (\Gamma_x^{5r_\ell/6} ,  (\Gamma_n \cup \Gamma_x) \setminus \Gamma_x^{5r_\ell/6}) \geq \frac{\eps_\ell^{\frac{1}{6}} k_\ell r_\ell  }{n}\right) \leq \exp \left\{-b {\dumep^{-1/6}}  \right\} \, ,
		\]
	where $b=(5e^2 2^{18})^{-1}$.
	\end{corollary}
	\begin{proof}
		Given that $\Capkl (\Gamma_x^{5r_\ell/6} ,  (\Gamma_n \cup \Gamma_x) \setminus \Gamma_x^{5r_\ell/6}) \geq \frac{\eps_\ell^{\frac{1}{6}} k_\ell r_\ell}{n}$, we can condition on $\Gamma_n \cup \Gamma_x$ and obtain intervals $(A_j)_{j=1}^{(2^{13}e)^{-1}\dumep^{-1/3}}$ on the graph $G_n/(\{\sunk\} \cup \Gamma_n \cup \Gamma_x)$ as described above. Applying \cref{lem:interval test tail Bj G'}, we then deduce that 
		\begin{align*}
			&\pr \left(  |\Bsun(x,r_\ell)| \leq 16 r_\ell^2 \eps_\ell  \middle\vert \Capkl (\Gamma_x^{5r_\ell/6} ,  (\Gamma_n \cup \Gamma_x) \setminus \Gamma_x^{5r_\ell/6}) \geq \frac{\eps_\ell^{\frac{1}{6}} k_\ell r_\ell  }{n}, \,\, \Gamma_n \cup \Gamma_x \right) \\
			&\leq \prod_{j=1}^{(2^{13}e)^{-1}\dumep^{-1/3}} \pr \left( \Bj \middle\vert \Bjj,  (\Gamma_n \cup \Gamma_x), \Capkl (\Gamma_x^{5\dumr/6}  ,  \Gamma_n \cup \Gamma_x \setminus \Gamma_x^{5\dumr/6}) \geq \frac{  \dumr \dumk \dumep^{\frac{1}{6}}}{n} \right) \\
			&\leq \left(1-\frac{1}{160e}\dumep^{1/6}\right)^{(2^{13}e)^{-1}\dumep^{-1/3}} \leq \exp \left\{-b {\dumep^{-1/6}}  \right\}.
		\end{align*}
		To conclude, we average over $\Gamma_n \cup \Gamma_x$, then transfer this result from $\Bsun(x,r_\ell)$ in $\UST(G_n/(\{\sunk\} \cup \Gamma_n \cup \Gamma_x)$ to $B_{\T_n}(x,r_\ell)$ in $\UST(G_n)$ using the stochastic domination result of \cref{lem:stoch dom add sun}, as explained above.
	\end{proof}

We now have all the tools to prove \cref{thm:lowermassboundUSTs}.

\begin{proof}[Proof of \cref{thm:lowermassboundUSTs}] Let $\delta >0$. We define the events 
\begin{align*}
		A_{\ell} 
		&= \left\{ \exists x \in \T_n: \left|B\left(x, r_\ell\right)\right| \leq {\epsilon_\ell r_\ell^2} \text{ and } \left|B\left(x, r_{\ell+1}\right)\right| \geq \epsilon_{\ell+1} r_{\ell+1}^2 \right\},  \\
		B_{\ell} 
		&= \left\{ \exists x \in \T_n: \left|B\left(x, r_\ell\right)\right| \leq \eps_\ell r_\ell^2 \right\} \, ,
		\end{align*}
for $\ell \in \{0,\ldots, N_n\}$. We decompose by writing
		\begin{align}\label{eqn:bootstrap 1}
		\pr\left( \exists x \in \T_n: |B(x, c \sqrt{n})| \leq \epsilon c^2 n\right) &\leq \pr\left( \neg \Enc \right) + \left( \sum_{\ell=0}^{N_n-1} \pr\left(\Enc \cap A_{\ell} \right) \right) + \pr\left(\Enc \cap B_{N_n} \right).
		\end{align}
In what follows we will show that given $\delta>0$ we can find $\eps$ and $c$ small enough and $N$ large enough so that the sum above is at most $3\delta$. This yields the required assertion of the theorem since the quantity $\pr(\exists x \in \T_n: |B(x, c \sqrt{n})| \leq \epsilon n)$ is non-decreasing as $c$ decreases.

We first apply \cref{thm:gammaisnice} and find $\eps$ and $c$ small enough and $N$ large enough (depending on $\delta$) that the first term is at most $\delta$ for all $n \geq N$. To control the second term in \eqref{eqn:bootstrap 1} we note that if $A_{\ell}$ occurs, then  $\left|B\left(v, r_{\ell+1}\right)\right| \leq \eps_{\ell} r_{\ell}^2=16\eps_{\ell+1} r_{\ell+1}^2$ for all $v \in B\left(x, r_{\ell+1}\right)$, and the number of such $v$ is at least $\eps_{\ell+1}r_{\ell+1}^2$. Therefore using \cref{thm:capacityofprefix}, \cref{cor:volLBfromCap} and Markov's inequality, we have for all $n$ large enough that
		\begin{align}
		\begin{split}\label{eqn:bootstrap 3}
		\sum_{l=0}^{N_n-1} \pr\left(\Enc \cap A_{\ell} \right) &\leq \sum_{l=0}^{N_n-1} \pr\left(\Enc \ \mathrm{and} \  \big | \left\{v \in \T_n: \left|B\left(v, r_{\ell+1}\right)\right| \leq 16\eps_{\ell+1} r_{\ell+1}^2 \right\} \big | \geq \eps_{\ell+1}r_{\ell+1}^2  \right) \\
		&\leq n \sum_{l=0}^{N_n-1} \eps_{\ell+1}^{-1}r_{\ell+1}^{-2}  \big ( e^{-a (\log \eps_\ell^{-1})^2} +  e^{-b\eps_\ell^{-\frac{1}{6}}} \big ) \, ,
		\end{split}
		\end{align}
By making $\eps$ smaller and $N$ larger if necessary we can guarantee that the term in the parenthesis on the right hand side is at most $\eps_\ell^{10}$ for all $n \geq N$. This shows that the sum can be smaller than $\delta$ as long as $\eps$ is small enough and $N$ is large enough. 

Finally, for the third term we recall that $r_{N_n} = c n^{\frac{1}{2} - \frac{\alpha}{10}}$ and $\eps_{N_n} = \eps n^{-\frac{\alpha}{5}}$, and use \cref{thm:capacityofprefix}, \cref{cor:volLBfromCap} and the union bound to bound
$$ \pr\left(\Enc \cap B_{N_n} \right) \leq n\left( e^{-a \log^2 (\eps^{-1}n^{\alpha/5})} + e^{-b\eps^{-\frac16}n^{\alpha/30} } \right) \, ,$$ 
which tends to $0$ as $n\to \infty$, so it is smaller than $\delta$ as long as $n$ is large enough. Provided $n$ is sufficiently large, we have therefore bounded \eqref{eqn:bootstrap 1} by $3\delta$, concluding the proof. (We can then reduce $\epsilon$ if necessary so that the bound holds for all $n\geq 1$).
\end{proof}

\section{Proofs of Theorems \ref{thm:gammaisnice} and \ref{thm:capacityofprefix}}\label{sctn:cap bounds}
In this section we prove \cref{thm:gammaisnice} and \cref{thm:capacityofprefix}. Due to the results of \cite{MNS}, this essentially boils down to proving only capacity estimates. In both cases, we will bound capacity using \cref{lem: var def of k-cap}. In \cref{sbsctn:general cap claims} we prove two claims that we later use in the proofs of \cref{thm:gammaisnice} and \cref{thm:capacityofprefix} in \cref{sbsctn:cap of Gamman} and \cref{sbsctn:cap of prefix} respectively.

\subsection{Two claims}\label{sbsctn:general cap claims}

In what follows we take $z = 1/20$ and assume that $\{G_n\}$ is a sequence of graphs satisfying \cref{assn:main}. Our first claim shows that with very high probability any loop-erased trajectory (that has bounded bubble-sum) has a rather long subinterval which is derived from a (relatively) short segment of a random walk trajectory (which in turn will have a long subinterval with good $M^{(k)}$ and closeness values by the subsequent claim).

\begin{claim}\label{claim: LE prefix to SRW interval2}
	Fix $\secp >0$ and $c>0$. There exists $\eps' >0$ such that for every $\eps \in(0,\eps')$ there exists $N$ such that for all $n\geq N$ and for all scales $\ell\in\{0,\ldots,N_n\}$ the following holds.
	Let $X$ be a random walk on $G_n$ which is bubble-terminated (see \cref{def:bubbleterm}) with bubble-sum bounded by $\secp$ and let $\Gamma$ be its loop erasure.
	Also fix $j\in\NN$ and $\K = \min\{\frac{z}{3\secp}, \frac{1}{24}\}$. Then with probability at least $1-\exp\left(-\frac{\dumep^{-\frac{z}{3}}}{\log(1/\dumep)}\right)$ either 
	$$ |\Gamma| < \frac{j\dumr}{24} \, ,$$
	or there exists $t \in[(j-1)\dumr/24,j\dumr/24]$ such that for all integers $1 \leq m \leq {\K}\dumep^{-\frac{2z}{3}}\log \dumep^{-1}$,
	\[
	\lambda_{t+m \dumep^{\frac{5z}{3}} \dumr}(X) - \lambda_{t+(m - 1) \dumep^{\frac{5z}{3}} \dumr}(X) \leq \dumep^{z}\dumr \, .
	\]
\end{claim}

\begin{proof}

	We set $\Mm=\dumep^{-5z/3}$. On the event that $|\Gamma| \geq j\dumr/24$, we divide $\Gamma[(j-1)\dumr/24,j\dumr/24]$ into $\Mm/24$ consecutive disjoint subintervals of length $\dumr/\Mm$. For $m \in \{1,\ldots, \Mm/24\}$ we say that the $m$-th interval is good if
\begin{equation*}
	\lambda_{\frac{(m+1)\dumr}{\Mm}(X)} - \lambda_{\frac{m\dumr}{\Mm}}(X) \leq \frac{\dumr}{\Mm^{3/5}} = \dumep^{z}\dumr \, .
\end{equation*}

As we assumed that the bubble sum is bounded by $\secp$, it follows from \cref{cl:indep:loops} that conditioned on $\Gamma$ and the event $\{|\Gamma| \geq j\dumr/24\}$, the collection of events that the $m$-th interval is good are independent. Furthermore, by \cref{cl:indep:loops}, \cref{claim: cap to bub UB} and Markov's inequality the probability of each such event is at least
\begin{equation*}
	1-\frac{\secp}{\Mm^{2/5}}.
\end{equation*} Hence, the probability that a sequence of ${\K}\dumep^{-\frac{2z}{3}} \log \dumep^{-1}$ disjoint consecutive intervals are all good is at least 
$$ (1- \secp \Mm^{-2/5})^{{\K}\dumep^{-\frac{2z}{3}} \log \dumep^{-1}} \geq \dumep^{2\K \secp} \, ,$$
where we used the inequality $1-x \geq e^{-2x}$ valid for $x>0$ small enough. Since there are $\frac{\Mm}{24}$ intervals in total, we can form $\frac{\Mm}{24}({\K}\dumep^{-\frac{2z}{3}} \log \dumep^{-1})^{-1}$ disjoint runs of ${\K}\dumep^{-\frac{2z}{3}} \log \dumep^{-1}$ consecutive intervals. Since the events are independent conditionally on $\Gamma$, we deduce that the probability that none of these runs contain only good intervals is at most
$$ \left(1-\dumep^{2 \K\secp}\right)
^{\frac{\Mm}{24} \big({\K}\dumep^{-\frac{2z}{3}}\log \dumep^{-1} \big)^{-1}}
\leq \exp\left(-\frac{\dumep^{2 \K\secp -5z/3}}{24\K \dumep^{-2z/3}\log\dumep^{-1}}\right) \leq
\exp\left(-\frac{\dumep^{2 \K\secp -z}}{24\K \log\dumep^{-1}}\right) \leq
\exp\left(-\dumep^{-{z \over 3}} \over \log \dumep^{-1}\right),
$$
where in the last inequality we used the fact that $ \K = \min\left\{{\frac{z}{3 \secp}},\frac{1}{24}\right\}$ by assumption. \end{proof}

For the next claim recall the definitions of $M^{(k)}$ in \eqref{def:Mk} and of $\close_k(U,V)$ in \eqref{def:close}. We show that with very high probability, any random walk interval of length of order $\dumep^z \log \dumep^{-1}\dumr$ has a slightly shorter subinterval of length of order $\dumep^z \dumr$, such that its value of $M^{(k)}$ and its closeness to the rest of the path are very close to their expected values given by \cref{lem:mk} and \eqref{eqn:expected closeness}. This is done by finding many well separated intervals and employing the fast mixing of the graph to obtain independence. 

		\begin{claim}\label{claim:X subinterval good M new} 
		Fix some $\K,c>0$. There exists $\eps' >0$ such that for every $\eps \in (0,\eps')$ there exists $N$ such that for all $n\geq N$ and for all scales $\ell\in\{0,\ldots,N_n\}$ the following holds.
			Let $X$ be a random walk on $G_n$, started from stationarity. 
	Let $M>0$ and fix some interval $I \subset [0,M\sqrt{n}]$ with $|I|= \frac{1}{2} \K \dumep^{z} \log \dumep^{-1}\dumr$. Also let $W \subset G_n$ be fixed. Then with probability at least $1-2e^{-\frac{\K z}{6}(\log \dumep^{-1})^{2}}$ there exists a subinterval $J = [t_J^-, t_J^+] \subset I$ such that  
		\begin{enumerate}[(1)]
			\item $|J| = 2 \dumep^{z}\dumr$,
			\item $M^{(\dumk)}(X[J]) \leq \frac{\dumr}{4}$
			\item $\close_{\dumk} \left(X[J], X\left[t_J^+ + \frac{\dumr}{24}, M\sqrt{n}\right] \cup X[0, t_J^- - \frac{\dumr}{24}]\right) \leq \frac{ \dumr \dumk^2 M}{n^{1.5}}$.
			\item  $\close_{\dumk} (X[J], W) \leq \frac{ \dumr \dumk^2 |W|}{n^{2}}$. 
		\end{enumerate}
	\end{claim}
	\begin{proof}
	We write $I=[t_I^-,t_I^+]$ and then further subdivide $I$ into $\frac{\K}{5}\log \dumep^{-1}$ segments of length $2\dumep^{z}\dumr$ separated by buffers of length $\frac{1}{2}\dumep^{z}\dumr$, that is, we set
		\begin{equation}\label{eqn:Cij def}
			I_j = [t_j^-, t_j^+) :=  {\left[t_I^- + \frac{5}{2}j\dumep^{z}\dumr + \frac{1}{4}\dumep^{z}\dumr,t_I^- + \frac{5}{2}(j+1)\dumep^{z}\dumr - \frac{1}{4}\dumep^{z}\dumr\right)}
		\end{equation}
			for each non-negative integer $j \leq \frac{\K}{5}\log \dumep^{-1}$. It will be important soon that the length of the buffers satisfy $\frac{1}{4}\dumep^{z}\dumr \geq \frac{1}{4}\eps_{N_n}^{z}r_{N_n} \gg n^{\frac{2\alpha}{3}} \tmix$ for all $n$ large enough by \cref{assn:main}, \eqref{eq:parameters1} and \eqref{eq:parameters2}. We also set 
		\[
			\Cavoidi =  X\left[0, t_I^- - \frac{\dumr}{36}\right) \cup X{\left[ t_I^+ + \frac{\dumr}{36}, M\sqrt{n}\right)}.
		\]
	We condition on $\Cavoidi$ and define for each $j$ the event
	\[
		\Engood = \left\{ M^{(\dumk)}(X[I_j]) \leq \dumr/4 \pand \close_{\dumk}\left(X[I_j], \Cavoidi \right) \leq \frac{\dumr \dumk^2 M}{n^{3/2}} \pand \close_{\dumk}\left(X[I_j], W \right) \leq \frac{\dumr \dumk^2|W|}{n^2}\right\}.
	\]

	Note that since $|I| = \frac{1}{2}\K \dumep^{z} \log \dumep^{-1}\dumr$, we have that $t_j^- - t_I^- \leq \frac{\dumr}{72}$ and $t_I^+ - t_j^+ \leq \frac{\dumr}{72}$ for all $j$ and all $\ell$ provided $\eps$ is small enough. Thus,
\begin{equation*}
	\close_{\dumk}(X[I_j],\Cavoidi) \geq \close_{\dumk}\left(X[I_j], X\left[0, t_j^- - \frac{\dumr}{24}\right] \cup X\left[t_j^+ + \frac{\dumr}{24}, M\sqrt{n}\right]  \right).
\end{equation*}
Hence $\Engood$ implies that the interval $I_j$ satisfies the conditions $(1)-(4)$. Note that the events $\{\Engood\}_j$ are not independent, but that was why we introduced the buffers. Let $\{\Ijbuf\}_{j \leq \frac{\K}{5}\log \dumep^{-1}}$ be independent random walks started from stationarity and run for time $2\dumep^{z}\dumr$, set
\[
		\EngoodBuf = \left\{ M^{(\dumk)}(\Ijbuf) \leq \dumr/4 \pand \close_{\dumk}\left(\Ijbuf, \Cavoidi \right) \leq \frac{\dumr \dumk^2 M}{n^{3/2}} \pand \close_{\dumk}\left(\Ijbuf, W \right) \leq \frac{\dumr \dumk^2|W|}{n^2}\right\} \, ,
	\] 
	and note that conditioned on $\Cavoidi$ the events $\{\EngoodBuf\}_j$ are independent. Now by \cref{lem:mk} and \cref{assn:main} we have (provided $\eps <1$ and $c< 1/6$, for example) that
	\begin{equation}\label{eq:expected_M} \nonumber
		\E\left[ M^{(\dumk)}(Y_j)\right] \leq 2\theta \dumep^{z}\dumr \, .
	\end{equation}
	Since $|\Cavoidi| \leq M\sqrt{n}$, it also follows from \eqref{eqn:expected closeness} that
	\begin{equation}\label{eq:closeness} \nonumber
		\E\left[\close_{\dumk} \left(\Ijbuf, \Cavoidi \right) \middle|\Cavoidi \right]\leq \frac{8\dumep^{z}\dumr \dumk^2 M}{n^{3/2}} \, ,
	\end{equation}
and that
\begin{equation}\label{eq:closeness:secondrw} \nonumber
	\E\left[\close_{\dumk} \left(\Ijbuf, W\right) \right]\leq \frac{8\dumep^{z}\dumr \dumk^2|W|}{n^{2}} \, .
	\end{equation}
Consequently, by Markov's inequality and independence we get that
\begin{align}\label{eqn:Eij occurs calc}
		\pr \left( \text{none of } \{\EngoodBuf\} \text{ occur}  \middle| \Cavoidi \right) \leq \left((16+8\theta)\dumep^{{z}}\right)^{\frac{\K}{5}(\log \dumep^{-1})} \leq e^{-\frac{\K z}{6}(\log \dumep^{-1})^{2}} \, ,
\end{align} 
as long as $\eps$ is small enough depending on $\theta$. To conclude, note that as long as $n$ is large enough, we can couple the independent walks $\{Y_j\}$ and $\{X[I_j]\}$ so that 
	\begin{align}\label{eqn:segment coupling}
		\pr\left(\exists j: X[I_j] \neq \Ijbuf \right) \leq  \frac{\K \log(1/\dumep)}{5}2^{-n^{\frac{2\alpha}{3}}} \leq e^{-\frac{\K z}{6}(\log \dumep^{-1})^{2}} \, .
	\end{align} 
Indeed, assume we coupled the first $j-1$ pairs and condition on all these pairs. By the Markov property and since the buffers between distinct $I_j$'s are longer than $n^{\frac{2\alpha}{3}}\tmix$, the starting point of $X[I_j]$ is 
$2^{-n^{\frac{2\alpha}{3}}}$ close in total variation distance to the stationary distribution by \eqref{eq:tvDecay}. Therefore, we may couple it to the first vertex of $\Ijbuf$ so that they are equal with probability at least $1-2^{-n^{\frac{2\alpha}{3}}}$ (see for instance \cite[Proposition 4.7]{levin2017markovmixing}). Moreover, once their starting points are coupled, we can run the walks together so that they remain coupled for the remaining $2\dumep^{z}\dumr$ steps. Hence \eqref{eqn:segment coupling} holds for large enough $n$ and we combine with \eqref{eqn:Eij occurs calc} in a union bound to conclude that 
\[ \pushQED{\qed} \pr \left( \text{none of } \{\Engood\} \text{ occur}  \middle| \Cavoidi \right) \leq 2 e^{-\frac{\K z}{6}(\log \dumep^{-1})^{2}} \, . 
\]
\end{proof}

\subsection{Proof of \cref{thm:gammaisnice}}\label{sbsctn:cap of Gamman}

	As mentioned in \cref{sctn:stoch dom lemmas}, to sample $\Gamma_n$ we will use a coupling with the sunny graph $G_n^* = G_n^*(\fp)$ introduced in \cite{PeresRevelleUSTCRT}, obtained from $G_n$ by adding an extra vertex $\rho_n$ known as the sun, and connecting it to every vertex in $v \in G_n$ with an edge of weight $\frac{ (\deg v) \fp}{\sqrt{n}-\fp}$ (so that the probability of jumping to $\rho_n$ at any step is $\fp n^{-1/2}$). 
	It follows from \cref{lem:stoch dom add sun} that the graph $\UST (G_n^*)\setminus\{\rho_n\}$ obtained from the UST of $G_n^*$ by removing $\rho_n$ and its incident edges is stochastically dominated by the UST of $G_n$.
	Therefore, there is a coupling between $\UST (G_n)$ and $\UST (G_n^*)$ such that $\UST (G_n^*)\setminus\{\rho_n\} \subset \UST (G_n)$; moreover if $\widetilde{\Gamma}_n^*$ denotes the path between $\ru$ and $\rv$ in $\UST (G_n^*)$, then $\Gamma_n = \widetilde{\Gamma}_n^*$ in this coupling provided that $\rho_n \notin \widetilde{\Gamma}_n^*$. 
	
	Note that this sunny graph is \textit{different} to the sunny graph used in the statements of \cref{lem:interval test tail Bj G'} and \cref{cor:volLBfromCap}. As outlined in \cref{sctn:bootstrap}, \cref{lem:interval test tail Bj G'} and \cref{cor:volLBfromCap} refer to later stages of the overall proof strategy.
	
	Consequently, it will be convenient to work with a path $\Gamma_n^* = \Gamma_n^*(\fp)$ sampled as follows.
	Let $T$ and $T'$ be two independent geometric random variables with mean $\fp^{-1} n^{1/2}$. Given $T$, let $X$ be a random walk run for $T-1$ steps started from $u \in G_n$ and let $\LE(X)$ be its loop erasure. Given $T'$ and $X$, run $X'$, a random walk started from $v \in G_n$ and terminate $X'$ after $T'$ steps. Write $T_X$ for the minimum between $T'$ and the first hitting time of $X$. Let $\Gamma_n^*$ be the path between $(u,v)$ in $\LE(X)\cup\LE(X'[0,T_X])$, if such a path exists. Otherwise, let $\Gamma_n^* = \emptyset$.

\begin{lemma}\label{prop: coupling with path}
	For every $\delta>0$ there exists $\fp>0$ such that for all large enough $n$ there exists a coupling of $\Gamma_n$ and $\Gamma_n^*(\fp)$ such that $\Gamma_n = \Gamma_n^*(\fp)$ and is non-empty with probability at least $1-\delta$. 
	\end{lemma}
	\begin{proof}
		For every path $\Gamma$ from $u$ to $v$ let $H(\Gamma) = \Gamma$ be equal to $\Gamma$ if $\rho_n\notin\Gamma$ and $H(\Gamma) = \emptyset$ if $\rho_n\in\Gamma$.
		Run Wilson's algorithm on the graph $G_n^*(\fp)$ initiated at the points $\rho_n, u$ and then $v$, and note that the hitting time of $\rho_n$ is a geometric random variable, and moreover that given $\tau_{\rho_n}$, the walk until time $\tau_{\rho_n}$ is distributed as a random walk on $G_n$. Consequently, $H(\widetilde{\Gamma}_n^*)$ has the distribution of $\Gamma_n^*$.
		
		By the discussion above, we can find a coupling of $(\Gamma_n, \widetilde{\Gamma}_n^*)$ where these paths are equal whenever $\rho_n \notin \widetilde{\Gamma}_n^*$. Under this coupling, we have that $(\Gamma_n, \widetilde{\Gamma}_n^*, H(\widetilde{\Gamma}_n^*))$ are all equal with probability $1-\pr(\rho_n \notin \widetilde{\Gamma}_n^*)$. As $H(\widetilde{\Gamma}_n^*)$ has the law of $\Gamma_n^*$, this is in fact a coupling of $\Gamma_n$ and $\Gamma_n^*$ where the paths are equal with probability $1-\pr(\rho_n \notin \widetilde{\Gamma}_n^*)$. By \cite[Claim 2.9]{MNS}, this probability tends to $1$ as $\fp\to 0$. For the final part of the claim, note that $\Gamma_n^*$ is clearly non-empty on this good event.
	\end{proof}

\begin{proof}[Proof of \cref{thm:gammaisnice}] Our main effort is to show that part (II) holds with high probability. Indeed, that part (I) and (III) occur with probability at least $1-\delta$ as long as $\eps>0$ and $c>0$ are small enough is a consequence of \cite[Theorem 1.1 and Theorem 2.1]{MNS}.

Let $\delta>0$. We appeal to \cref{prop: coupling with path} and obtain $\zeta>0$ so that $\pr\left(\Gamma_n \neq \Gamma_n^*(\fp)\right)\leq \delta/4$. Denote by $\mathcal{B}$ the event of part (II) of \cref{thm:gammaisnice}. For the rest of the proof, we think of $\delta$ and $\zeta$ as fixed, we set $\secp = \theta + 2\zeta^{-2}$ and $\K = \min\{\frac{z}{3\secp}, \frac{1}{24}\}$, and decrease both $c$ and $\eps$ until we eventually obtain that $\pr(\mathcal{B}^c)\leq \delta$. Recall that $\Gamma_n^*(\fp)$ is generated using two independent random walks with geometric killing time which we denote $X$ and $X'$. Setting $M = \frac{8}{\fp\delta}$, we can thus write
		\begin{equation*}
			\pr(\mathcal{B}^c) \leq \pr\left(\Gamma_n \neq \Gamma_n^*(\fp)\right) + \pr\left(|X|+|X'| \geq M\sqn \right) + \pr\left(|X|+|X'| \leq M\sqrt{n} \pand \mathcal{B}^c\right).
		\end{equation*}
		The first event has probability at most $\frac{\delta}{4}$ by the above.
		Since $M = \frac{8}{\fp\delta}$, the probability of the second event is also bounded by $\delta/4$ by Markov's inequality. For the third term, first decrease $c$ if necessary so it is less than $\frac{1}{2M}$ (this will be useful at the end of the proof), then let $\ell\leq N_n$ be a fixed scale and let $I\subset \Gamma_n^*(\fp)$ be some segment with $|I| = \dumr/3$. It therefore has at least $\dumr / 6$ vertices either on $\LE(X)$ or on $\LE(X')$ and hence contains at least one interval of the form $\LE(X)[(j-2)\dumr/24, (j+1)\dumr/24]$ or $\LE(X')[(j-2)\dumr/24, (j+1)\dumr/24]$ (that is, an interval of the form $[(j-1)\dumr/24, j\dumr/24]$ plus two buffers of length $\frac{\dumr}{24}$ both before and after the interval) for some $j \leq {24 M \sqrt{n} \over \dumr}$.

		Since \cref{assn:main} holds, we deduce from \cref{claim:bubblezeta} that $X$ and $X'$ are bubble-terminated random walks with bubble sum bounded by $\secp$. Hence we may apply \cref{claim: LE prefix to SRW interval2} and the union bound to learn that the probability that there exists a scale $\ell$ and $j$ as above such that the event of \cref{claim: LE prefix to SRW interval2} does not hold for $X$ or $X'$ is at most
		\begin{equation*}
	\sum_{\ell=0}^{\infty}\frac{M\sqrt{n}}{\dumr/24}\exp\left(-\frac{\dumep^{-\frac{z}{3}}}{\log(1/\dumep)}\right) 
	= \frac{24M}{c} \sum_{\ell=0}^{\infty} 2^\ell \cdot \exp\left(-\frac{\dumep^{-\frac{z}{3}}}{\log(1/\dumep)}\right) 
	= \frac{24M}{c}\sum_{\ell=0}^{\infty} 2^\ell \cdot \exp\left(-\frac{\eps^{-\frac{z}{3}} 4^{\frac{\ell z}{3}}}{\log(4^\ell \eps^{-1})}\right) \, ,
		\end{equation*}
		which can be made to be smaller than $\delta/4$ by decreasing $\eps$ appropriately. Thus we may assume without loss of generality that $I$ contains an interval of the form $\LE(X)[(j-2)\dumr/24, (j+1)\dumr/24]$ for some $j$ that we fix henceforth, and that there exists a time $t\in[(j-1)\dumr/24, j\dumr/24]$ such that for all integers $m$ satisfying $1 \leq m \leq \K \dumep^{-\frac{2z}{3}}(\log \dumep^{-1})$ we have
		\begin{equation}\label{eqn:special t def 1}
		\lambda_{t+m \dumep^{\frac{5z}{3}} \dumr}(X) - \lambda_{t+(m - 1) \dumep^{\frac{5z}{3}} \dumr}(X) \leq \dumep^{z}\dumr.
		\end{equation}

		We write $X[t_1,t_2)$ for the corresponding part of $X$, that is, we set $t_1 = \lambda_{t}(X)$ and $t_2 = \lambda_{t+\K\dumep^{z}\log(\dumep^{-1}) \dumr}(X)$. It holds by construction that
		\begin{equation}\label{eqn:t1-t2 bounds}
			t_2 - t_1 \geq \K \dumep^{z} \log \dumep^{-1} \dumr \quad \textrm{and} \quad t_2 \leq M\sqrt{n} \, . 
		\end{equation}

		We now apply the union bound using \cref{claim:X subinterval good M new} with $W=X'[0, M\sqn]$ to get that the probability that there exists a scale $\ell$ and $i\leq \frac{2M\sqrt{n}}{\K \dumep^z\log\dumep^{-1}\dumr}$ such that $(1)-(4)$ of \cref{claim:X subinterval good M new} do \emph{not} hold for the interval $\left[\left(i-1\right) \frac{1}{2}\left(\K \dumep^z\log\dumep^{-1}\dumr\right), i \frac{1}{2}\left(\K \dumep^z\log\dumep^{-1}\dumr\right)\right]$ is at most 
		\begin{equation*}
		\sum_{\ell=0}^{\infty} \frac{2M\sqn}{ \K \dumep^z\log\dumep^{-1}\dumr} \cdot \exp\left(-  \frac{\K z}{6} \left(\log \dumep^{-1}\right)^{2}\right) 
		\leq \sum_{\ell=0}^{\infty} \frac{2M\left(2\cdot4^z\right)^{\ell}} {\K \eps^z\log\eps^{-1}c} 
		\left(\frac{\eps}{4^\ell}\right)^{\log\left(\frac{4^\ell}{\eps}\right) \frac{\K z}{6}},
		\end{equation*}
		which can be made smaller than $\delta/4$ by decreasing $\eps$ appropriately. Therefore we henceforth assume that all such intervals contain a good subinterval satisfying $(1)-(4)$ of \cref{claim:X subinterval good M new}.

		Since $[t_1, t_2]$ must contain an interval of the form $\left[\left(i-1\right) \frac{1}{2}\left(\K \dumep^z\log\dumep^{-1}\dumr\right), i \frac{1}{2}\left(\K \dumep^z\log\dumep^{-1}\dumr\right)\right]$ by \eqref{eqn:t1-t2 bounds}, it now follows that $[t_1,t_2]$ contains a subinterval $J=[t_J^-,t_J^+]$ satisfying conditions $(1)-(4)$ of \cref{claim:X subinterval good M new} with $W=X'[0, M\sqn]$. Since $$|J| = 2\dumep^{z}\dumr \geq \lambda_{t+m \dumep^{\frac{5z}{3}} \dumr}(X) - \lambda_{t+(m - 2) \dumep^{\frac{5z}{3}} \dumr}(X)$$ for each $2 \leq m\leq \K \dumep^{-\frac{2z}{3}}(\log \dumep^{-1})$ by \eqref{eqn:special t def 1}, there must exist some $m^* \leq \K \dumep^{-\frac{2z}{3}}(\log \dumep^{-1})$ such that
		\[
		t_J^- \leq \lambda_{t+(m^* - 1) \dumep^{\frac{5z}{3}} \dumr}(X) < \lambda_{t+m^* \dumep^{\frac{5z}{3}} \dumr}(X) \leq t_J^+.
		\]
		We set $A=(\LE(X))_{[t+(m^* - 1) \dumep^{\frac{5z}{3}} \dumr, t+m^* \dumep^{\frac{5z}{3}} \dumr)}$, so that $A \subset I \subset \Gamma_n$, so that $|A|=\dumep^{\frac{5z}{3}} \dumr$ and so that $A \subset X[t_J^-, t_J^+]$. Since $M^{(\dumk)}(A)$ defined in \eqref{def:Mk} is monotone in $A$ we have that $M^{(\dumk)}(A) \leq M^{(\dumk)}(X[t_J^-, t_J^+]) \leq \dumr/4$ by $(2)$ of \cref{claim:X subinterval good M new}. Hence by \cref{lem: var def of k-cap},
		\begin{equation*}
		\Capkl(A) \geq \frac{\dumk|A|^2}{2nM^{(\dumk)}(A)} \geq \frac{2\dumk\dumr \dumep^{\frac{10z}{3}}}{n} \, .
		\end{equation*}
		Due to the buffers of length $\dumr/24$ present in the beginning and ending of $I$ it follows that $\left(\Gamma_n^* \setminus I\right) \subseteq X[0, t_J^- - \dumr/24] \cup X\left[t_J^+ + \dumr/24, M\sqn\right] \cup W$. Hence by $(3)-(4)$ of \cref{claim:X subinterval good M new} we get
		\begin{align*}
		\close_{\dumk} \left(A,  \Gamma_n^* \setminus I \right) &\leq \close_{\dumk} \left(X[t_J^-, t_J^+], X[0, t_J^- - \dumr/24] \cup X\left[t_J^+ + \dumr/24, M\sqn\right] \cup X' \right) \\
		&\leq \frac{2\dumr \dumk^2M}{ n^{3/2}}
		= \frac{2\dumep^{\frac{10z}{3}}\dumr \dumk}{n} \cdot \frac{M\dumep^{\frac{1}{2}-\frac{10z}{3}}\dumr}{\sqrt{n}} \, ,
		\end{align*}
		where we used $\dumk = \dumep^{1/2}\dumr$ and $|X| + |X'| \leq M\sqn$.	Consequently, since we chose $c<\frac{1}{2M}$ and $z=1/20$, we can reduce $\eps$ if necessary so that
		\begin{align*}
		\Capkl\left(I, \Gamma_n^* \setminus I\right) \geq \Capkl\left(A, \Gamma_n^*\setminus I\right) 
		&\geq \Capkl(A) - \close_{\dumk} (A, \Gamma_n^* \setminus I) \\
		&\geq \frac{2\dumk\dumr \dumep^{\frac{10z}{3}}}{n} \left( 1 -\frac{M\dumep^{\frac{1}{2}-\frac{10z}{3}}\dumr}{\sqrt{n}} \right) \geq \frac{\dumk\dumr \dumep^{\frac{10z}{3}}}{n},
		\end{align*}
		as required. Finally, to cover the case where $I$ is primarily contained in $\LE(X')$ rather than $\LE(X)$, note that we can reverse the roles of $X$ and $X'$ above to obtain a fourth contribution to the probability of $\frac{\delta}{4}$. This concludes the proof.
		\end{proof}

\subsection{Proof of \cref{thm:capacityofprefix}}\label{sctn:capacity tail decay}\label{sbsctn:cap of prefix}

We assume that $\Enc$ holds and let $\ell\leq N_n$ be a fixed scale throughout the proof. The proof will involve applications of \cref{claim: LE prefix to SRW interval2} and  \cref{claim:X subinterval good M new}; for these we will take $\secp = \theta + \frac{1}{c^2}$, take $\K=\min\{\frac{z}{3\secp}$, $\frac{1}{24}\}$, take $\Mm = \dumep^{-1/10}$ and take $W = \Gamma_n$. These four variables will assume these values throughout the proof.
\begin{proof}
Let $x$ be some vertex of $G_n$ and let $X$ be a random walk started from $x$ and let $\tau_{\Gamma_n}$ denote the time at which $X$ hits $\Gamma_n$, so that $\Gamma_x = \LE(X[0, \tau_{\Gamma_n}])$. We start by upper bounding the time until $X$ hits $\Gamma_n$. On the event $\Enc$ we have from \cref{thm:gammaisnice} (I) that $\Capp_{\sqrt{n}} (\Gamma_n) \geq 2c$. It therefore follows from \cref{cl:cap:mix} that for each $i \geq 1$, 
	\[
	\pr \left( X[(i-1) \sqrt{n}, i\sqrt{n}) \cap \Gamma_n \neq \emptyset \middle| X[0, (i-1) \sqrt{n}) \cap \Gamma_n = \emptyset \right) \geq \frac{2c}{3}.
	\]
Consequently, taking a product over $i \leq \Mm$ it follows that
\begin{equation}\label{eqn:hitting time UB}
\pr \left( \tau_{\Gamma_n} > \Mm\sqrt{n} \right) \leq \left(1 - \frac{2c}{3} \right)^{\Mm} \leq e^{-2c\Mm/3}.
\end{equation}
Provided that $\eps$ is small enough as a function of $c$, this is much smaller than the required bound on the probability of \cref{thm:capacityofprefix}; hence we work on the event $\{\tau_{\Gamma_n} \leq \Mm\sqrt{n}\}$ for the rest of the proof. Furthermore, if $\{|\Gamma_x| \leq \frac{\dumr}{3}\}$, then $\Gamma_x^{5\dumr/6}$ contains a segment $I \subset \Gamma_n$ with $|I| = \frac{\dumr}{3}$ (see the definitions above \cref{thm:capacityofprefix}). Hence on the event $\Enc$ it follows from \cref{thm:gammaisnice} (II) (which we just proved in the previous subsection) that
		\[
			\Capp_{\dumk}\left(\Gamma_x^{5\dumr/6}, (\Gamma_n \cup \Gamma_x) \setminus \Gamma_x^{5\dumr/6}\right) \geq \Capp_{\dumk} (I, \Gamma_n \setminus I) \geq \frac{\dumep^{1/6} \dumk \dumr}{n},
		\]
		so that the tail bound of \cref{thm:capacityofprefix} holds. We therefore also assume that $\{|\Gamma_x| > \frac{\dumr}{3}\}$.

		Under $\Enc$ we have that $\Capp_{\sqn}(\Gamma_n) \geq 2c$ and hence by \cref{claim: cap to bub UB} we have that $B_{\Gamma_n}(G_n) \leq \theta + \frac{1}{c^2}=\secp$, so $X$ is bubble-terminated random walk with bubble sum bounded by $\secp$. Now divide $X([0, \Mm \sqrt{n}])$ into $2\Mm\sqrt{n}( \K \dumep^z\log\dumep^{-1}\dumr)^{-1}$ disjoint consecutive intervals of length $\frac{1}{2} \K \dumep^z\log\dumep^{-1}\dumr$. Also note that $2\Mm\sqrt{n} \leq \K \dumep^z\log\dumep^{-1}e^{\frac{\K z}{12}(\log \dumep^{-1})^{2}}\dumr$ for all $\ell \leq N_n$ provided that $\eps$ is small enough as a function of $\K$ and $c$ (i.e., depending on $\theta$ and $c$). By the union bound and \cref{claim:X subinterval good M new}, provided $n$ exceeds some $N(c, \eps)$ the probability that all of these consecutive intervals contain a subinterval satisfying points $(1)-(4)$ of \cref{claim:X subinterval good M new} is therefore at least
		\[
1 - 2\Mm\sqrt{n}( \K \dumep^z\log\dumep^{-1}\dumr)^{-1}e^{-\frac{\K z}{6}(\log \dumep^{-1})^{2}} \geq 1 - e^{\frac{\K z}{12}(\log \dumep^{-1})^{2}} e^{-\frac{\K z}{6}(\log \dumep^{-1})^{2}} = 1 - e^{-\frac{\K z}{12}(\log \dumep^{-1})^{2}}.
		\]
In particular, since \textit{any} interval $I \subset [0, \Mm\sqrt{n}]$ of length $\K \dumep^z\log\dumep^{-1}\dumr$ must contain an entire consecutive interval of the form above, we deduce that, provided $n \geq N(c, \eps)$,
\begin{equation}\label{eqn:clm 4.1 good event}
\pr \left(\forall I \subset [0, \Mm\sqrt{n}], |I| = \K \dumep^z\log\dumep^{-1}\dumr: \exists J \subset I \ \mathrm{ satisfying } \ (1)-(4) \ \mathrm{ of \ \cref{claim:X subinterval good M new}} \right) \geq  1 - e^{-\frac{\K z}{12}(\log \dumep^{-1})^{2}}.
\end{equation}
		 
		 We next apply \cref{claim: LE prefix to SRW interval2} with $j=1$ to obtain that, provided $n \geq N(c, \eps)$, with probability at least 
\begin{equation}\label{eqn:claim 4.2 prob}
		1- \exp\left(-\frac{\dumep^{-\frac{z}{3}}}{\log(1/\dumep)}\right)
\end{equation}
		there exists $t \leq \frac{\dumr}{3}$ such that for all $1 \leq m \leq \K \dumep^{-\frac{2z}{3}}(\log \dumep^{-1})$, 
		\begin{equation}\label{eqn:special t def 2}
		\lambda_{t+m \dumep^{\frac{5z}{3}} \dumr}(X) - \lambda_{t+(m - 1) \dumep^{\frac{5z}{3}} \dumr}(X) \leq \dumep^{z}\dumr\, .
		\end{equation}
		We write $X[t_1,t_2)$ for the corresponding part of $X$, so that $t_1 = \lambda_{t}(X)$ and $t_2 = \lambda_{t+\K\dumep^{z}\log(\dumep^{-1}) \dumr}(X)$. It holds by construction that
		\begin{equation}\label{eqn:t1-t2 bounds 2}
		\K \dumep^{z} \log \dumep^{-1} \dumr \leq t_2 - t_1 \, , 
		\end{equation}
		and moreover since we assumed that $\{|\Gamma_x| > \frac{\dumr}{3}\}$ and $\{\tau_{\Gamma_n} \leq \Mm\sqrt{n}\}$, we clearly have that $t_2 \leq \Mm\sqrt{n}$. On the event $\Enc$, it therefore follows from \eqref{eqn:clm 4.1 good event} and \eqref{eqn:t1-t2 bounds 2} that the probability that $[t_1,t_2]$ does not contain a subinterval $J=[t_J^-,t_J^+]$ satisfying conditions $(1)-(4)$ of \cref{claim:X subinterval good M new} is bounded by $e^{-\frac{\K z }{12} (\log \dumep^{-1})^2}$.
		
For the rest of the proof we assume that such a $J$ exists and that $\Enc$ holds. By part (1) of \cref{claim:X subinterval good M new} and \eqref{eqn:special t def 2},
\[
|J| = 2\dumep^{z}\dumr \geq \lambda_{t+m \dumep^{\frac{5z}{3}} \dumr}(X) - \lambda_{t+(m - 2) \dumep^{\frac{5z}{3}} \dumr}(X)
\]
for each $2 \leq m\leq \K \dumep^{-\frac{2z}{3}}(\log \dumep^{-1})$. Therefore there must exist some $m^* \leq \K \dumep^{-\frac{2z}{3}}(\log \dumep^{-1})$ such that
		\[
		t_J^- \leq \lambda_{t+(m^* - 1) \dumep^{\frac{5z}{3}} \dumr}(X) < \lambda_{t+m^* \dumep^{\frac{5z}{3}} \dumr}(X) \leq t_J^+.
		\]
		Now set $A=(\Gamma_x)_{[t+(m^* - 1) \dumep^{\frac{5z}{3}} \dumr, t+m^* \dumep^{\frac{5z}{3}} \dumr)}.$
		
		Note that, by construction, it holds that $A \subset \Gamma_x^{\frac{\dumr}{3}}$, that $|A|=\dumep^{\frac{5z}{3}} \dumr$ and that $A \subset X[t_J^-, t_J^+]$. Since $M^{(\dumk)}(A)$ as defined in \eqref{def:Mk} is monotone with respect to $A$ we have that $M^{(\dumk)}(A) \leq M^{(\dumk)}(X[t_J^-, t_J^+]) \leq \dumr/4$ by $(2)$ of \cref{claim:X subinterval good M new}. Hence by \cref{lem: var def of k-cap},
		\begin{equation*}
		\Capkl(A) \geq \frac{\dumk|A|^2}{2nM^{(\dumk)}(A)} \geq \frac{2\dumk\dumr \dumep^{\frac{10z}{3}}}{n} \, .
		\end{equation*}

Since $t_J^+ \leq \lambda_{\dumr/3}(X)$ by construction, and $W=\Gamma_n$, it also follows that $$(\Gamma_n \cup \Gamma_x) \setminus \Gamma_x^{5r_\ell/6} \subset W \cup  X[t_J^+ + \dumr/24, \Mm\sqn].$$ Therefore, since $\close_{\dumk}(\cdot, \cdot)$ is monotone and subadditive in each argument (by definition and the union bound), applying $(3)-(4)$ of \cref{claim:X subinterval good M new} we deduce that
		\begin{align*}
		\close_{\dumk} \left(A,  (\Gamma_n \cup \Gamma_x) \setminus \Gamma_x^{5r_\ell/6} \right) &\leq \close_{\dumk} \left(X[t_J^-, t_J^+], X\left[t_J^+ + \dumr/24, \Mm\sqn\right] \right) + \close_{\dumk} \left(X[t_J^-, t_J^+], W \right) \\
		&\leq \frac{\dumr \dumk^2 (\Mm\sqn + |W|)}{ n^{2}}
		\leq \frac{2\dumep^{\frac{10z}{3}}\dumr \dumk}{n} \cdot \frac{\Mm\dumep^{\frac{1}{2} - \frac{10z}{3}}\dumr}{\sqrt{n}} \, ,
		\end{align*}
		(where we used $\dumk = \dumep^{1/2}\dumr$ and $|W| \leq \dumep^{-1/10} \sqn = M_{\ell}\sqn$ on the event $\Enc$). Consequently, since $z=1/20$, recalling that $\Mm=\dumep^{-1/10}$ and assuming without loss of generality that $c, \eps < 1/2$, we obtain that
		\begin{align*}
		\Capkl \left(\Gamma_x^{\dumr/3},  (\Gamma_n \cup \Gamma_x) \setminus \Gamma_x^{5r_\ell/6} \right) &\geq \Capkl \left(A,  (\Gamma_n \cup \Gamma_x) \setminus \Gamma_x^{5r_\ell/6} \right) \\
		&\geq \Capkl(A) - \close_{\dumk}  \left(A,  (\Gamma_n \cup \Gamma_x) \setminus \Gamma_x^{5r_\ell/6} \right) \\
		&\geq \frac{2\dumk\dumr \dumep^{\frac{10z}{3}}}{n} \left( 1 -\frac{\dumep^{\frac{7}{30}}\dumr}{\sqrt{n}} \right) \geq \frac{\dumk\dumr \dumep^{\frac{1}{6}}}{n}.
		\end{align*}
To summarize, we showed that $\Capkl \left(\Gamma_x^{\dumr/3},  (\Gamma_n \cup \Gamma_x) \setminus \Gamma_x^{5r_\ell/6} \right)$ is large enough on the event $\Enc$ whenever $\{\tau_{\Gamma_n} \leq \Mm\sqrt{n}\}$ and the relevant events in \cref{claim: LE prefix to SRW interval2} and \cref{claim:X subinterval good M new} occur so that we can find $A$ as above. \cref{thm:capacityofprefix} therefore follows on taking a union bound over \eqref{eqn:hitting time UB}, \eqref{eqn:clm 4.1 good event} and \eqref{eqn:claim 4.2 prob}, choosing $\epsilon'$ small enough as a function of $c$ and requiring that $n $ is large enough as a function of $\eps$ and $c$ (since $\K$ and $\secp$ were themselves functions of $c$).
\end{proof}

	\section{Proof of Lemma \ref{lem:interval test tail Bj G'}}\label{sctn:conditional tail bound}\label{sctn:proof on conditional sunny graph}

	In this section we prove \cref{lem:interval test tail Bj G'}. Throughout we assume that the index $n$, the scale $\ell$ and the paths $\Gamma_n \cup \Gamma_x$ are fixed. We also take the setup of Section \ref{sctn:bootstrap}, as outlined above \cref{lem:interval test tail Bj G'}. This means that we condition on $\Gamma_n \cup \Gamma_x$ and add a sun $\sunk$ to the graph $G_n/(\Gamma_n \cup \Gamma_x)$ with weights chosen so that a lazy random walk will jump to the sun at the next step with probability $\frac{1}{\dumk}$. We also assume that the intervals $A_j \subset \Gamma_x$ for $j=1,\ldots, (2^{13}e)^{-1}\dumep^{-\frac13}$ are predefined as described in \cref{sctn:bootstrap}. 
	For the rest of this section we work on the graph $G_n/(\{\sunk\}\cup \Gamma_n \cup \Gamma_x)$.	Recall that
	\begin{equation}\label{eq:parameters3} 
		r_\ell = \frac{r}{2^\ell}, \quad \epsilon_\ell = \frac{\epsilon}{4^\ell}, \quad k_\ell = \eps_\ell^{1/2}\dumr,  \quad |A_j| \leq 2^{13}e\dumep^{1/3}\dumr, \quad \Capkl (A_j, (\Gamma_n \cup \Gamma_x) \setminus A_j) \geq \frac{2^{11}e \dumep \dumr^2}{n}\,.
	\end{equation}
	When we talk about capacity and relative capacity in this section, we are always referring to these quantities on the \textit{original graph} $G_n$.

Recall also that, for each $j \leq (2^{13}e)^{-1} \dumep^{-\frac13}$, we let $I_j(\dumk)$ be the set of vertices connected to the contracted vertex in $\UST(G_n/(\{\sunk\} \cup \Gamma_n \cup \Gamma_x)$ by a path of length at most $\dumk$, such that the last edge on this path has an endpoint in $A_j$. This also includes vertices originally in $A_j$ before the contraction. Since $\ell$ is fixed for this section, we also set $X_j = |I_j(\dumk)|$.
		
	
	\begin{claim}\label{cl: cap LB}
		Assume that $\Gamma_x$ and $\Gamma_n$ satisfy \eqref{eq: good paths event} (and therefore \eqref{eq:parameters3}). Fix a scale $\ell$ and consider the graph $G_n / (\Gamma_n\cup \Gamma_x \cup \{\sunk\})$ as described above.
		Then, for every $j\in \{1,\ldots, (2^{13}e)^{-1} \dumep^{-\frac13}\}$ we have 
		\begin{equation*}
			\E[X_j] \geq n\cdot \frac{\Capkl(A_j, (\Gamma_n \cup \Gamma_x)\setminus A_j)}{2e} \geq 2^{10}\dumep\dumr^2.
		\end{equation*}
		
	\end{claim}
	\begin{proof}
		By Wilson's algorithm, for every $v\in G_n$, we have that $v\in I_j(\dumk)$ if a random walk starting at $v$ hits $\Gamma_n \cup \Gamma_x \cup \{\sunk\}$ at $A_j$ and its loop erasure is of length at most $\dumk$. Therefore,
		\begin{equation*}
			\pr(v\in I_j) \geq \pr_v(\tau_\sunk > \dumk)\cdot\pr_v(\tau_{A_j} < \dumk \pand \tau_{A_j} < \tau_{(\Gamma_n \cup \Gamma_x) \setminus A_j} \mid \tau_\sunk > \dumk),
		\end{equation*}
		where all hitting times refer to hitting times of the lazy random walk.
		First note that $\pr(\tau_\sunk > \dumk) = \left(1-\frac{1}{\dumk} \right)^{\dumk} \geq \frac{1}{2e}$. Then, given $\tau_\sunk > \dumk$, the lazy random walk until time $\dumk$ is distributed as a lazy random walk on $G_n / (\Gamma_n \cup \Gamma_x)$.
		Since all degrees in $G_n$ are equal we get 
		\begin{align*}
			\E[X_j] = \sum_{v\in G_n}\pr(v \in I_j) &\geq \sum_{v\in G_n}\frac{\pr_v(\tau_{A_j} < \dumk \pand \tau_{A_j} < \tau_{(\Gamma_n \cup \Gamma_x) \setminus A_j} \text{in } G_n/(\Gamma_n\cup\Gamma_x)) }{2e} \\&=  n\cdot \frac{\Capkl(A_j, (\Gamma_n \cup \Gamma_x)\setminus A_j)}{2e} \, ,\qedhere
		\end{align*} 
		and we conclude the proof using \eqref{eq:parameters3}.
	\end{proof}

 	Recall that our goal is to find a lower bound for the probability that $\sum_{i=1}^{j+1} X_i$ is large given that $\sum_{i=1}^j X_i$ is small. To this end, let $\Phi_j$ be the (random) edge-set consisting of all simple paths of length at most $\dumk$ in $\UST(G_n / (\{\sunk\} \cup \Gamma_n \cup \Gamma_x))$ that end in the contracted vertex through $A_1 \cup \ldots \cup A_j$. Note that $\Phi_j$ determines $\{\sum_{i=1}^j X_i \leq 16\dumep\dumr^2\}$ and that conditioning on $\Phi_j = \Wj$ for some set of edges $\Wj$ means precisely that the edges of $\Wj$ are in the $\UST$ (open edges) and all other edges touching a vertex $v$ of $\Wj$, such that the path in $\Wj$ from $v$ to $A_1 \cup \ldots \cup A_j$ is of length at most $\dumk-1$, must not belong to the $\UST$ (closed edges). These open and closed edges determine $\Phi_j$. Thus, to condition on $\Phi_j = \Wj$, we erase the closed edges and contract all the open edges to a single vertex which coincides with $\Gamma_n \cup \Gamma_x \cup \{\sunk\}$, and call the remaining graph $\Gj$. By the spatial Markov property of the $\UST$ \cite[Proposition 4.2]{benjamini2001} we have that $\UST(\Gj)$ together with $\varphi_j$ is distributed precisely as $\UST(G_n / (\{\sunk\} \cup \Gamma_n \cup \Gamma_x))$ conditioned on $\Phi_j = \varphi_j$. Note that the event $\{\sum_{i=1}^j X_i \leq 16\dumep\dumr^2\}$ occurs if and only if $|V(\varphi_j)|\leq 16\dumep\dumr^2$ where $V(\varphi_j)$ are the vertices touched by $\varphi_j$.




	
	

	\begin{claim}\label{claim:RW hitting prob UB on G'}
		Let $\Wj \subset E(G_n)$ be such that $\pr \left( \Phi_j = \Wj \right)>0$ and $|V(\Wj)| \leq 16\dumep\dumr^2$. Let $\gamma$ be a simple path in $\Gj$ that ends at the contracted vertex. Let $(Y_t)_{t \geq 0}$ denote a lazy random walk on $\Gj$ started from a uniform vertex $U$ of the original graph $G_n$ and killed upon hitting the contracted vertex of $\Gj / \gamma$, that is, the upon hitting the vertex corresponding to the contracted edges $\{\sunk\} \cup \Gamma_n \cup \Gamma_x \cup \Wj \cup \gamma$. Denote by $V(\Gamma_n \cup \Gamma_x \cup \Wj \cup \gamma)$ the set of vertices of $G_n$ touched by the edges in $\Gamma_n \cup \Gamma_x \cup \Wj \cup \gamma$ and let $M \subset V(\Gamma_n \cup \Gamma_x \cup \Wj \cup \gamma)$ be a fixed subset of vertices of $G_n$. Then 
	\begin{align*}
		\pr \left(Y \,\,\mathrm{ hits }\,\, M\right) \leq \frac{64\dumep \dumr^2}{n} + \frac{4|M|\dumk}{n} \, .
	\end{align*}
	(Recall here that to ``hit $M$" means to hit the contracted vertex via an edge that originally led to $M$).
\end{claim}	
\begin{proof} 
	Let $\Delta = \deg (G_n)$, i.e. the degree of vertices in the original graph $G_n$ (recall that by \cref{assn:main} all vertex degrees are equal), and let 
	\[
	\V = \left\{v \in V(G_n) \setminus V(\Gamma_n \cup \Gamma_x \cup \Wj): \deg_{\Gj} (v) \leq \frac{\Delta}{2}\right\} \, .
	\]
	In other words, $\V$ is the set of all vertices of $G_n$ that are not in the contracted vertex of $\Gj$ that are adjacent to at least $\Delta/2$ closed edges. Since $|V(\Wj)|\leq 16\dumep\dumr^2$, the number of closed edges is no more than $16 \Delta \dumep\dumr^2$. Hence the number of vertices touching a closed edge is at most $32 \Delta \dumep\dumr^2$ and each vertex in $\V$ contributes at least $\Delta/2$ to this count, so $|\V|\leq 64\dumep\dumr^2$.


	Recall that, when we originally added the sun to $G_n / (\Gamma_n \cup \Gamma_x)$, we chose the weights so that the probability that a lazy random walk on $G_n / (\Gamma_n \cup \Gamma_x)$ would jump to the sun at the next step is always $\frac{1}{\dumk}$. In the graph $\Gj$, we have now contracted some edges and closed some other edges. For any $x \in \Gj$, these operations can only increase the probability that $Y$ will jump directly to the sun from the vertex $x$. Therefore, by coupling, we can separate the sun and its incident edges, and obtain an upper bound for $\pr \left(Y \,\,\mathrm{ hits }\,\, M\right)$ by instead bounding the same probability for a lazy random walk on $(\Gj/\gamma) \setminus \{ \sunk \}$ with an independent \textsf{Geo}$(\frac{1}{\dumk})$ killing time. We denote this second lazy random walk by $Y'$.
	
	To control capacity on $(\Gj/\gamma) \setminus \{ \sunk \}$ we will need to work with the stationary measure on $(\Gj/\gamma) \setminus \{ \sunk \}$, which we denote by $\pi'$. (The bound on $|\V|$ above will then help us to compare $\pi'$ with the uniform measure). We define $\pi'$ on all of $G_n$ by remembering the edges from before the contraction. In particular, this means that for $u \in G_n$, we have
	\[
\pi' (u) = \frac{\Delta - N^{\mathrm{cl}}(u)}{\sum_{v \in G_n}(\Delta - N^{\mathrm{cl}}(v))},	
	\]
	where $N^{\mathrm{cl}}(v)$ denotes the number of closed edges incident to $v$ in $\Gj$.

We now observe the following. If $u \in G_n \setminus \V$, then
	\begin{align*}
		\pi'(u) \geq \frac{\Delta/2}{n\Delta} \geq \frac{1}{2n}.
	\end{align*}
	Also, for every $u \in G_n$, provided that $c<1/32$ and $\eps<1$, we have that
	\begin{align*}
		\pi'(u) \leq \frac{\Delta}{n\Delta - 32 \Delta \dumep \dumr^2} \leq \frac{2}{n}.
	\end{align*}
	In what follows, these two observations mean that we will be able to switch between $\pi'$ and $U$ and vice versa provided we multiply by $2$. In particular, we can write
	\begin{align*}
		\pr_{U} \left(Y' \text{ hits } M\right) \leq \pr \left(U \in \V \right) + 2\pr_{\pi'}\left( Y' \text{ hits } M \right) &\leq \frac{64\dumep \dumr^2}{n} + 2\sum_{t=0}^{\infty} \pr_{\pi'} \left(Y_t' \in M\right) \\
		&= \frac{64\dumep \dumr^2}{n} + 2\sum_{t=0}^{\infty} \pi' (M) \pr\left( \textsf{Geo}\left(\frac{1}{\dumk}\right) \geq t \right) \\
		&\leq \frac{64\dumep \dumr^2}{n} + \frac{4|M|}{n} \sum_{t=0}^\infty \left( 1 - \frac{1}{\dumk}\right)^t =  \frac{64\dumep \dumr^2}{n} + \frac{4|M|\dumk}{n}.\qedhere
	\end{align*}
\end{proof}

We will use \cref{claim:RW hitting prob UB on G'} to prove the following upper bounds.
	
	\begin{lemma}\label{lem:exp and var UB G'}
		Let $\Wj \subset E(G_n)$ be such that $\pr \left( \Phi_j = \Wj \right)>0$ and $|V(\Wj)| \leq 16\dumep\dumr^2$. Then
		\begin{enumerate}[(i)]
			\item $\E\left[ X_{j+1}  \mid \Phi_j = \Wj \right] \leq 5 \cdot 2^{13} \cdot e \cdot \dumep^{5/6}\dumr^2$.
			\item $\var \left(  X_{j+1} \mid \Phi_j = \Wj \right) \leq 68\dumep\dumr^2 \E\left[X_{j+1}\mid \Phi_j = \Wj\right]$, 
		\end{enumerate}
	\end{lemma}
	\begin{proof} We condition on $\Phi_j = \Wj$ throughout this proof so our probability space is that of $\UST(\Gj)$. To prove (i) we condition on $\Phi_j = \Wj$ and take any $v \in \Gj \setminus \{\sunk\}$. By Wilson's algorithm on the graph $\Gj$, we have that $\pr \left(v \in A_{j+1} \mid \Phi_j = \Wj\right)$ is upper bounded by the probability that a lazy random walk started at $v$ hits $A_{j+1}$ before it hits the sun. If $(Y_t)_{t \geq 0}$ is such a random walk starting from a uniform vertex of $G_n$, by \cref{claim:RW hitting prob UB on G'} and \eqref{eq:parameters3} we have that
			\begin{align*}
				\E\left[ X_{j+1} \mid \Phi_j = \Wj \right] &\leq n\pr \left(Y_t \text{ hits } A_{j+1}\right) \leq 64\dumep \dumr^2 + 4|A_{j+1}|\dumk \leq 5|A_{j+1}|\dumk \leq 5 \cdot 2^{13} \cdot e \cdot \dumep^{5/6}\dumr^2 \, ,
				%
			\end{align*}
			where we also used the upper bound on $|A_{j+1}|$ in \eqref{eq:parameters3}.

			To ease notation in the proof of (ii) we write $\pr(\cdot), \E\left[\cdot\right]$ and $\Var(\cdot)$ for $\pr\left(\cdot \mid \Phi_j = \Wj \right)$ and the corresponding expectation and variance. 
			We have 
			\begin{align}\label{eq: cond var calc}
				\var(X_{j+1}) &= \sum_{u,v \in G_n} \pr(u,v \in I_{j+1}) - \pr(u\in I_{j+1})\pr(v\in I_{j+1}).
				\\& = \sum_v \sum_u \Big [\pr\left(u \in I_{j+1} \mid v \in I_{j+1}\right) - \pr(u\in I_{j+1})\Big]\pr(v\in I_{j+1}). \nonumber
			\end{align}
			Fix some $v$, and rewrite the inner sum as
			\begin{align*}
				 n\big [ \pr(U\in I_{j+1} \mid v\in I_{j+1}) - \pr(U \in I_{j+1})\big ],
			\end{align*}
				where $U$ is a vertex chosen uniformly from $G_n$.
			We decompose the event $v\in I_{j+1}$ according to $\gamma_v$, the path from $v$ to $A_{j+1}$ in $\Gj$ which is of length at most $\dumk$ and obtain that
			\begin{align*}
				&\pr(U\in I_{j+1} \mid v\in I_{j+1}) - \pr(U \in I_{j+1}) \\&= \sum_{\gamma_v}\pr(\gamma_v \subseteq \UST(\Gj) \mid v\in I_{j+1})\left[\pr(U\in I_{j+1} \mid \gamma_v \subseteq \UST(\Gj)) - \pr(U\in I_{j+1})\right].
			\end{align*}
			To compare $\pr(U\in I_{j+1} \mid \gamma_v\subseteq \UST(\Gj))$ and  $\pr(U\in I_{j+1})$ we note again from the spatial Markov property \cite[Proposition 4.2]{benjamini2001} that the rest of $\UST(\Gj)$ given $\gamma_v\subseteq \UST(\Gj)$ is the $\UST$ on the graph obtained from $\Gj$ by contracting $\gamma_v$. By coupling Wilson's Algorithm running on each of the two graphs ($\Gj$ and $\Gj/\gamma_v$), the difference between the two quantities can be upper bounded by the probability that a random walk starting from a uniform vertex of $G_n$ hits $\gamma_v$ before it hits the new sun $\sunk$.
			By \cref{claim:RW hitting prob UB on G'}, this is bounded by $\frac{64\dumep \dumr^2}{n} + \frac{4\dumk^2}{n}$ uniformly for all $\gamma_v$ with $|\gamma_v| \leq \dumk$. 
			As $\sum_{\gamma_v}\pr(\gamma_v \subseteq \UST(\Gj) \mid v\in I_{j+1})$ sums to $1$ we obtain that
			\begin{equation*}
				n\left(\pr(U\in I_{j+1} \mid v\in I_{j+1}) - \pr(U \in I_{j+1})\right) \leq 64\dumep\dumr^2 + 4\dumk^2.
			\end{equation*}
			Plugging this into \eqref{eq: cond var calc} and using \eqref{eq:parameters3} we obtain 
			\begin{equation*}
				\var(X_{j+1}) \leq \sum_v 	(64\dumep\dumr^2 + 4\dumk^2) \pr(v\in I_{j+1}) \leq (64\dumep\dumr^2 + 4\dumk^2) \E[X_{j+1}] = 68\dumep\dumr^2 \E[X_{j+1}].\qedhere
			\end{equation*}
		
	\end{proof}

Recall that $\Bj = \{ \sum_{i=1}^j X_i \leq 16\dumep \dumr^2\}$, and $\Phi_j$ is the random edge-set induced by $\cup_{i=1}^j I_i (\dumk)$. Under $\Bj$, we have no information about the structure of $\Phi_j$, other than that $|\Phi_j| \leq 16\dumep \dumr^2$ (and this was important for the factorization in the proof of \cref{cor:volLBfromCap}). However, in order to prove \cref{lem:interval test tail Bj G'}, we will need the following lower bound.

	\begin{lemma}\label{lem: some structures are good}
		It holds that
		\begin{equation*}
			\pr \left(\Phi_j \in \{\Wj : \E[X_{j+1} | \Phi_j = \Wj] \geq 2^9\dumep\dumr^2 \} \mid \Bj\right) \geq \frac{\dumep^{1/6}}{80e}.
		\end{equation*}
	\end{lemma}
	\begin{proof}
		Recall that we are working on the graph $G_n/(\Gamma_n\cup \Gamma_x \cup \{\sunk\})$. Suppose that $\sum_{i=1}^j X_i \leq 16\dumep \dumr^2$, and note that this event can be written as the disjoint union of all possible $\Wj$ such that $\pr(\Phi_j = \Wj)>0$ and $|V(\Wj)|\leq 16\dumep \dumr^2$.
		When conditioning on $\Phi_j = \Wj$ for some $\Wj$ we work on the graph $\Gj$, as defined above \cref{claim:RW hitting prob UB on G'}. Note that by \cref{lem:exp and var UB G'}, we have that for every $\Wj$ with $|V(\Wj)| \leq 16\dumep \dumr^2$ and $\pr \left( \Phi_j = \Wj \right)>0$ that 
		\begin{equation*}
			\E\left[X_{j+1} | \Phi_j = \Wj \right]  \leq 5 \cdot 2^{13} e \dumep^{5/6}\dumr^2. 
		\end{equation*} 
		Furthermore, by \cref{cl: expected stoch dom} and \cref{cl: cap LB} we have that
		\begin{equation*}
			\E\left[X_{j+1} \middle\vert \sum_{i=1}^j X_i \leq 16\dumep \dumr^2 \right] \geq \E\left[X_{j+1}\right] \geq 2^{10}\dumep\dumr^2.
		\end{equation*}
		Write $\E'$ and $\pr'$ for the expectation and probability operators $\pr(\cdot \mid \Bj)$ and $\E\left[\cdot \mid \Bj\right]$ on $G_n / (\{\sunk\} \cup \Gamma_n \cup \Gamma_x)$.
		We have that
		\begin{align*}
			\E'\left[X_{j+1} \middle\vert \Phi_j \right] &\leq 5 \cdot 2^{13} \cdot e\cdot\dumep^{5/6}\dumr^2\quad \text{a.s.},
			\\\E'\left[\E'\left[X_{j+1} \mid \Phi_j\right]\right] &= \E'[X_{j+1}] \geq 2^{10}\dumep\dumr^2.
		\end{align*} 
		Therefore
		\begin{equation*}
			2^{10}\dumep\dumr^2 \leq \E'[X_{j+1}] \leq 2^9\dumep\dumr^2 + \pr'[\E'[X_{j+1} \mid \Phi_j] \geq 2^9\dumep\dumr^2] \cdot 5 \cdot 2^{13} \cdot e \dumep^{5/6}\dumr^2.
		\end{equation*}
Rearranging, we deduce that
		\begin{equation*}
			\pr'\left(\E'[X_{j+1} \mid \Phi_j] \geq 2^9\dumep\dumr^2\right) \geq \frac{\dumep^{1/6}}{80e},
		\end{equation*}
		as required.
	\end{proof}
	
	\begin{lemma}\label{lem: cheby for good wj}
		Suppose $\Wj$ is such that $\pr \left( \Phi_j = \Wj \right)>0$ and $\E\left[X_{j+1} \mid \Phi_j = \Wj\right] \geq 2^9\dumep \dumr^2$. Then 
		\[
		\pr\left( X_{j+1} \leq 16\dumep \dumr^2 \mid \Phi_j = \Wj \right) \leq \frac{1}{2}.
		\]
	\end{lemma}
	\begin{proof}
		The result is a straightforward application of Chebyshev's inequality, similarly to \cite[Lemma 6.13]{Hutchcroft2020universality}. First note that it follows from \cref{lem:exp and var UB G'}(ii) that
		\begin{equation*}
			\Var \left(  X_{j+1} \mid \Phi_j = \Wj \right) \leq
			68\dumep\dumr^2\E\left[ X_{j+1} \mid \Phi_j = \Wj \right] \leq \frac{68}{2^{9}}\E\left[ X_{j+1} \mid \Phi_j = \Wj \right]^2,
		\end{equation*}
		where in the last inequality we used that $\E\left[X_{j+1} \middle\vert \Phi_j = \Wj\right] \geq 2^{9}\dumep \dumr^2$ by assumption. Using this again we therefore deduce that
		\begin{align*}
			\pr\left( X_{j+1} \leq 16\dumep \dumr^2 \middle\vert \Phi_j = \Wj \right) \leq \pr\left( X_{j+1} \leq \frac{1}{2^{5}}\E\left[X_{j+1} \middle\vert \Phi_j = \Wj\right] \middle\vert \Phi_j = \Wj \right) \leq \frac{2 \Var \left(  X_{j+1} \middle\vert \Phi_j = \Wj  \right)}{\E\left[ X_{j+1} \middle\vert \Phi_j = \Wj \right]^2} \leq \frac{1}{2}.
		\end{align*}
	\end{proof}
	\begin{proof}[Proof of \cref{lem:interval test tail Bj G'}]
		By \cref{lem: some structures are good}, given that $\sum_{i=1}^j X_i \leq 16\dumep\dumr^2$, we get with probability at least $\dumep^{1/6}/80e$ that $\Wj$ satisfies
		\begin{equation*}
			\E[X_{j+1} \mid \Phi_j=\Wj] \geq 2^9\dumep\dumr^2.
		\end{equation*}
		For every such $\Wj$, by \cref{lem: cheby for good wj}, we get that given $\Phi_j = \Wj$, we have that $X_{j+1} \geq 16\dumep\dumr^2$ with probability at least $1/2$. We conclude that
		\begin{equation*}
			\pr \left( \Bj \middle\vert \Bjj,  (\Gamma_n \cup \Gamma_x), \Capkl (\Gamma_x^{5\dumr/6}  ,  \Gamma_n \cup \Gamma_x \setminus \Gamma_x^{5\dumr/6}) \geq \frac{  \dumr \dumk \dumep^{\frac{1}{6}}}{n} \right) \leq 1 - \frac{\dumep^{1/6}}{160e},
		\end{equation*}
		as required.
	\end{proof}
	
\section{A criterion for GHP convergence}\label{sec:abstract}

\subsection{GP convergence}

We first aim to address the convergence provided in \cref{thm:pr04}. Recall our definitions and notation from \cref{sec:defs} (in fact this section can be seen as a direct continuation of \cref{sec:defs}).
	
\begin{definition} \label{def:GP} Let $(X,d,\mu)$ and $(X',d',\mu')$ be elements of $\XX_c$. The \textbf{Gromov-Prohorov (GP) pseudo-distance} between $(X,d,\mu)$ and $(X',d',\mu')$ is defined as
		\[
		\dGP((X,d,\mu),(X',d',\mu')) = \inf \left\{d_P(\phi_* \mu, \phi_*' \mu') \right\},
		\]
		where the infimum is taken over all isometric embeddings $\phi: X \rightarrow F$, $\phi': X' \rightarrow F$ into some common metric space $F$.
\end{definition}

Thus $\dGP$ is a metric on $\XX_c^{\GP}$ which is the space $\XX_c$ where we identify all mm-space with $\GP$ distance $0$. There is a useful equivalent definition of convergence of mm-spaces with respect to the $\GP$ distance. Given an mm-space $(X,d,\mu)$ and a fixed $m \in \NN$ we define a measure $\nu_m((X,d,\mu))$ on $\RR^{m \choose 2}$ to be the law of the ${m \choose 2}$ pairwise distances between $m$ i.i.d.~points drawn according to $\mu$.

\begin{theorem}[Theorem 5 in \cite{GrevenPfaffWinterGwConvergence}] Let $(X_n,d_n,\mu_n)$ and $(X,d,\mu)$ be elements of $\XX_c^{\GP}$. Then 
$$ \dGP( (X_n,d_n,\mu_n), (X,d,\mu)) \longrightarrow 0 \, ,$$
if and only if for any $m\in \NN$ 
$$ \nu_m((X_n,d_n, \mu_n)) \Rightarrow \nu_m((X,d,\mu)) \, ,$$
where $\Rightarrow$ denotes standard weak convergence of measures on $\RR^{m \choose 2}$. 
\end{theorem}

This is still not quite the setting of this paper since $\UST$s are \textit{random} mm-spaces. Thus let $\M_1(\XX_c^{{\mathrm{GP}}})$ denote the space of probability measures on $\XX_c^{\mathrm{GP}}$. Each element $\pr\in \M_1(\XX_c^{\mathrm{GP}})$ therefore defines random measures $\left(\nu_m\right)_{m\geq 2}$ and we additionally have annealed measures on $\mathbb{R}^{\binom{m}{2}}$, given by
	\begin{equation*}
	\tilde{\nu}_m(\pr) := \int_{\XX_c^{\mathrm{GP}}} \nu_m((X,d,\mu))d\pr
	\end{equation*}
for each integer $m\geq 2$. It is often more straightforward to prove deterministic weak convergence of the measures $\tilde{\nu}_m$ for each $m\geq 2$, rather than distributional weak convergence of the random measures $\nu_m$ for each $m$. For example, the conclusion of \cref{thm:pr04} can be restated as
\begin{equation}\label{eq:prrestated} \tilde{\nu}_m \left(\left(\UST(G_n), d_n/(\beta_n \sqrt{n}), \mu_n\right)\right) \Rightarrow \tilde{\nu}_m(\CRT) \, ,\end{equation}
for any fixed $m\geq 2$. However, this does not immediately imply that the $\UST$s converge to the $\CRT$ in distribution with respect to the topology of $(\XX_c^{\mathrm{GP}},d_{\mathrm{GP}})$, e.g.~see \cite[Example 2.12 (ii)]{GrevenPfaffWinterGwConvergence}. Indeed, the random mm-spaces need not be tight. It is not hard to show that this is not the case in our setup. 

\begin{lemma}\label{lem:PRgromov}
		Suppose that $(G_n)_{n \geq 1}$ is a sequence of graphs satisfying Assumption \ref{assn:main}. Let $d_n$ denote the graph distance on $\UST(G_n)$ and $\mu_n$ the uniform probability measure on its vertices. Then there exists a sequence $(\beta_n)_n$ satisfying $0<\inf_n \beta_n \leq \sup_n \beta_n < \infty$ such that $(\UST(G_n),\frac{1}{\beta_n \sqrt{n}}d_n, \mu_n)$ converges in distribution to the $\CRT$ with respect the topology of $(\XX_c^{\mathrm{GP}},d_{\mathrm{GP}})$.
\end{lemma}
\begin{proof} We appeal to \cite[Corollary 3.1]{GrevenPfaffWinterGwConvergence} and verify conditions (i) and (ii) there. Condition (ii) is precisely \eqref{eq:prrestated}. To verify condition (i) we use \cite[Theorem 3]{GrevenPfaffWinterGwConvergence} (and recall that by Prohorov's Theorem the relative compactness of the measures is equivalent to their tightness) and verify conditions (i) and (ii) there (see also Proposition 8.1 in \cite{GrevenPfaffWinterGwConvergence}). Condition (i) is just saying that $\tilde{\nu}_2$ is a tight sequence of measures on $\RR$, which follows from \eqref{eq:prrestated}. Lastly, \cref{thm:lowermassboundUSTs} directly implies condition (ii) \cite[Theorem 3]{GrevenPfaffWinterGwConvergence}.
\end{proof}

We remark that the use of \cref{thm:lowermassboundUSTs} in the last line of the proof above is an overkill and it is not too difficult to verify condition (ii) of \cite[Theorem 3]{GrevenPfaffWinterGwConvergence} directly.

	\subsection{GHP convergence and the lower mass bound}
	The key to strengthening the $\GP$ convergence of \cite{PeresRevelleUSTCRT}, as stated in \cref{lem:PRgromov}, to $\GHP$ convergence is the lower mass bound criterion of \cite{AthreyaLohrWinterGap}. In \cite[Theorem 6.1]{AthreyaLohrWinterGap} it is shown that $\GP$ convergence of deterministic mm-spaces together with this criterion is equivalent to $\GHP$ convergence. In this paper we require an extension to the setting of random mm-spaces (i.e., measures on mm-spaces); it is not hard to obtain this using the ideas of \cite{AthreyaLohrWinterGap} and we provide it here (\cref{prop:lower mass bound}).

	As in \cite[Section 3]{AthreyaLohrWinterGap}, given $c > 0$ and an mm-space $(X,d,\mu)$ we define
	\begin{align*}
	m_c((X,d,\mu)) &= \inf_{x \in X}\{\mu (B(x, c))\} \, .
	\end{align*}
	
	We begin with a short claim about deterministic mm-spaces.
	
	\begin{claim}\label{claim: lower mass bound at inf}
		Let $(X_n,d_n,\mu_n)$ be a sequence of mm-spaces that is $\GP$-convergent to $(X,d,\mu)$, i.e., 
		$$\dGP( (X_n,d_n,\mu_n), (X,d,\mu)) \to 0 \, .$$ Suppose further that for any $c>0$ we have
		\begin{equation*}
		\inf_n m_c((X_n,d_n,\mu_n)) > 0 \, .
		\end{equation*}
		Then, for every $\eps>0$
		\begin{equation*}
		\inf_{x\in \supp(\mu)} \mu(B(x,\eps)) \geq \liminf_{n \to \infty}\inf_{x\in X_n} \mu_n\left(B(x,\eps/2)\right) > 0.
		\end{equation*}	
	\end{claim}
	
	\begin{proof}
		
		Fix some $x\in \supp(\mu)$ and $\eps>0$. Then $\mu(B(x,\eps/4)) \geq b$ for some $b=b(x,\eps)>0$. Put $\delta = \min\{b/2,\eps/12\}$.
		By the $\GP$ convergence there exists $N\in \NN$ such that for every $n\geq N$ there are isometric embeddings taking $X_n$ and $X$ to a common metric space $(E,d_n')$ such that the Prohorov distance between the pushforwards of their measures is smaller than $\delta$. Therefore we may assume that $X_n$ and $X$ are both subsets of some common metric space. We abuse notation and write $\mu$ and $\mu_n$ in place of their respective pushforward measures. Since the $\GP$ distance is at most $\delta$ we get that
		\begin{equation*}
		b\leq \mu(B(x,\eps/4)) \leq \mu_n(B(x,\eps/4+\delta))+\delta.
		\end{equation*}
		for $n\geq N$. Hence by our choice of $\delta$ we get
		\begin{equation*}
		\mu_n(B(x,\eps/3)) > 0.
		\end{equation*}
		Therefore, we can find some $y_n\in X_n$ such that $d_n'(x,y_n) < \eps/3$. Also, for any $\delta'\in(0,\eps/6)$, we can find $N_2 \in \NN$ such that for $n\geq N_2$ we have
		\begin{equation*}
		\inf_{y\in X_n} \mu_n(B(y,\eps/2)) \leq \mu_n(B(y_n,\eps/2)) \leq \mu(B(y_n,\eps/2+\delta')) + \delta' \leq \mu(B(x,\eps)) + \delta'.
		\end{equation*} 
	Hence, taking the $\liminf$ on the left hand side and then taking $\delta' \to 0$ we obtain that for all $x\in X$
		\begin{equation*}
		\liminf_{n \to \infty}\inf_{y\in X_n} \mu_n(B(y,\eps/2)) \leq  \mu(B(x,\eps)),
	\end{equation*}
	and the claim follows by taking the infimum over $x\in X$.
	\end{proof}
	
We now state and prove the main goal of this section; as we state immediately afterwards, it readily shows that \cref{thm:lowermassboundUSTs} implies \cref{thm:maingeneral}.

	\begin{theorem}\label{prop:lower mass bound}
		Let $((X_n,d_n,\mu_n))_{n\geq 1}, (X,d,\mu)$ be random mm-spaces and suppose that 
		\begin{enumerate}[(i)]
			\item $(X_n, d_n, \mu_n) \overset{(d)}{\longrightarrow} (X,d,\mu)$ with respect to the $\mathrm{GP}$ topology.
			\item For any $c > 0$, the sequence $\left(m_c((X_n,d_n,\mu_n))^{-1}\right)_{n \geq 1}$ is tight.
		\end{enumerate}
		Then $(X_n, d_n, \mu_n) \overset{(d)}{\to} (\supp(\mu),d,\mu)$ with respect to the $\GHP$ topology.
	\end{theorem}

	\begin{proof} The metric space  $(\XX_c,\dGP)$ is separable (see \cite[Figure 1]{AthreyaLohrWinterGap}), hence by the Skorohod Representation theorem, there exists a probability space on which the convergence in $(i)$ holds almost surely. We will henceforth work on this probability space, and may therefore assume that $(X_n)_{n\geq 1}$ and $X$ are embedded in a common metric space where $d_P(X,X_n) \to 0$ almost surely. We will show that on this probability space, we have that $(X_n, d_n, \mu_n) \longrightarrow (\supp(\mu),d,\mu)$ in \emph{probability} with respect to the $\GHP$ topology, giving the required assertion. 
	
	Let $\eps, \eps_2 >0$.
		By $(ii)$, we have that there exists some $c_1>0$ and $N_1\in \NN$ such that for every $n\geq N_1$ we have
		\begin{equation*}
		\pr\left(\inf_{x\in X_n} \mu_n(B_{d_n}(x,\eps/2)) \leq  c_1\right) \leq \eps_2.
		\end{equation*}
		Hence by Fatou's lemma
		\begin{equation*}
		\pr\left(\limsup_{n} \left\{\inf_{x\in X_n} \mu_n(B_{d_n}(x,\eps/2)) > c_1 \right\} \right) \geq 1-\eps_2.
		\end{equation*}
		Meaning, with probability larger than $1-\eps_2$, we can find a (random) subsequence $n_k$ such that for every $k\in \NN$ we have that 
		\begin{equation*}
		m_{\eps/2}((X_{n_k},d_{n_k},\mu_{n_k})) = \inf_{x\in X_{n_k}} \mu_{n_k}(B_{d_{n_k}}(x,\eps/2)) > c_1.
		\end{equation*}
		Hence by Claim \ref{claim: lower mass bound at inf}, on this event we have that $\inf_{x\in \supp(\mu)} \mu(B_d(x,\eps)) \geq c_1$. Next, since almost sure convergence implies convergence in probability, we get by assumption (i) that
		\begin{equation*}
		\lim_{n\to\infty}\pr\left( d_P\left((X_n, d_n, \mu_n),(X,d,\mu)\right) > \eps \wedge \frac{c_1}{2} \right) = 0. 
		\end{equation*}
		Hence we can find $N_2\in \NN$ such that for every $n\geq \max\{N_1,N_2\}$ with probability at least $1-3\eps_2$ the three events
		\begin{equation*}
		\inf_{x\in X_n} \mu_n(B_{d_n}(x,\eps)) \geq c_1  \,\,, \,\, \inf_{x\in \supp(\mu)} \mu(B_d(x,\eps)) \geq c_1 \,\, , \,\, d_P \left((X_n, d_n, \mu_n),(X,d,\mu)\right) \leq \eps \wedge \frac{c_1}{2} \, ,
		\end{equation*}
		occur. Let $x \in \supp(\mu)$. Since the Prohorov distance between $\mu$ and $\mu_n$ is smaller than $\eps \wedge \frac{c_1}{2}$, we have that 
		\begin{equation*}
		c_1 \leq \mu(B_d(x,\eps)) \leq \mu_n(B_{d_n}(x,2\eps)) + \frac{c_1}{2}.
		\end{equation*}
		Hence 
		\begin{equation*}
		\mu_n(B_{d_n}(x,2\eps)) > 0.
		\end{equation*}
		Thus, $\supp(\mu) \subseteq X_n^{2\eps}$. We use the same argument to obtain that under this event, $X_n \subseteq \supp(\mu)^{2\eps}$ and conclude that
		\begin{equation*}
		\pr\left(\dGHP\left((X_n, d_n, \mu_n),(\supp(\mu),d,\mu)\right) > 2\eps \right) \leq 3\eps_2.
		\end{equation*}
		We therefore get that $(X_n,d_n,\mu_n)$ converges in probability (hence, in distribution) to $(\supp(\mu),d,\mu)$ in the Gromov-Hausdorff-Prohorov topology, as required.
	\end{proof}

	\noindent {\bf Proof of \cref{thm:maingeneral}}.
	\cref{lem:PRgromov} shows that the $\UST$ sequence converges in distribution with respect to $\dGP$ to the CRT $(X,d,\mu)$ so that condition (i) of \cref{prop:lower mass bound} holds. \cref{thm:lowermassboundUSTs} verifies that condition (ii) holds, and lastly, it is well known (see \cite[Theorem 3]{AldousCRTI}) that $\supp(\mu)=X$. The conclusion of \cref{prop:lower mass bound} thus verifies \cref{thm:maingeneral}. \qed 

	\section{Comments and open questions}\label{sec:the_end}

Combining with self-similarity of the CRT, Theorem 1.1 can also be used to recover the UST scaling limit in other settings. For instance, Theorem 2 of \cite{aldous1994recursive} entails that the branch point between three uniformly chosen points in the CRT splits the CRT into three smaller copies of itself, with masses distributed according to the Dirichlet$(\frac{1}{2}, \frac{1}{2}, \frac{1}{2})$ distribution, and where each copy is independent of the others after rescaling. This together with \cref{thm:main1} shows the following.

\begin{example}
Set $G_n = \ZZ_{\lfloor n^\frac{1}{d}\rfloor}^d$, the torus on (approximately) $n$ vertices and $d>4$. Sample a Dirichlet$(\frac{1}{2}, \frac{1}{2}, \frac{1}{2})$ random variable, that is, a uniform triplet $(\Delta_1, \Delta_2, \Delta_3)$ on the $2$-simplex. Conditioned on this, let $G_{\lfloor \Delta_1 n \rfloor}$, $G_{\lfloor \Delta_2 n \rfloor}$, and $G_{\lfloor \Delta_3 n \rfloor}$ be disjoint and attach each to an outer vertex of a $3$-star. Let $T_n$ be the $\UST$ on the resulting graph and $\mu_n$ the uniform measure on its vertices. 	
Then $(T_n, \frac{1}{\beta (d) \sqrt{n}} d_{T_n}, \mu_n) \overset{(d)}{\longrightarrow} (\T,d_{\T},\mu)$.
\end{example}

Next, building on the corollaries in \cref{sctn:corollaries of main thm}, one can also ask finer questions about the structure of the UST in the mean-field regime. One in particular is the convergence of the height profile.

\begin{open}
Take the setup of \cref{thm:maingeneral}, and set $H_n(r) = \#\{v \in G_n: d_{\T_n}(O,v) = r \}$. Does the process $\left(H_n(r\beta_n\sqrt{n})/\sqrt{n}\right)_{r > 0}$ converge to its continuum analogue on the CRT? (That is, the Brownian local time process $(\ell(r))_{r \geq 0}$ defined in \cite[Theorem 1.1]{drmota1997profile}).
\end{open}

This does not follow straightforwardly from the GHP convergence of \cref{thm:maingeneral} since that only captures the convergence of full balls of diameter $\sqrt{n}$ with volumes of order $n$. (On the other hand, it is straightforward prove convergence of the rescaled volume profile $V_n(r) = \sum_{s \leq r} H_n(s)$ from GHP convergence).

Next, our paper addresses the general mean-field case but leaves the upper critical dimension case of $\ZZ_{n^{1/4}}^4$ open. Here the mixing time is really of order $n^{1/2}$, but it was shown by Schweinsberg \cite{schweinsberg} that Gromov-weak convergence to the CRT still holds with an additional scaling factor of $(\log n)^{1/6}$. Our proof of the lower mass bound does not immediately transfer to the $4$-dimensional setting. However, it is possible that it is attainable to do so using the recent results of Hutchcroft and Sousi \cite{hutchcroftsousi}.


	\begin{open}\label{open:4d}
	 Let $\T_n$ be a uniformly drawn spanning tree of the $4$-dimensional torus $\ZZ_n^4$. Denote by $d_{\T_n}$ the corresponding graph-distance in $\T_n$ and by $\mu_n$ the uniform probability measure on the vertices of $\T_n$. Let $\gamma_n$ be the sequence appearing in \cite[Theorem 1.1]{schweinsberg}, uniformly bounded away from $0$ and infinity. Does the lower mass bound of \cref{prop:lower mass bound}(ii) hold for the sequence $\left(\T_n,\frac{d_{\T_n}}{\gamma_n n^{2}(\log n)^{1/6}} ,\mu_n\right)_{n \geq 1}$?
	\end{open}

Finally, one may also ask whether USTs rescale to the CRT under the weaker assumptions of \cite{MNS}, under which the authors prove that the sequence of rescaled UST diameters is tight. In particular, they do not assume transitivity but instead require that the graph is \emph{balanced}; that is, there exists a constant $D< \infty$ such that $$\frac{\max_{v \in G_n} \deg v}{\min_{v \in G_n} \deg v} \leq D$$ for all $n$. It is straightforward to extend the proof of \cref{thm:lowermassboundUSTs} to this setting by carrying the constant $D$ through all our computations, but since we are still restricted by the assumption of transitivity of \cite{PeresRevelleUSTCRT} for \cref{thm:maingeneral} we have chosen to keep the notation simple and have not pursued this here.

\bibliographystyle{abbrv}
\bibliography{biblio}

\bigskip
\noindent \textsc{Department of Mathematical Sciences, Tel Aviv University, Tel Aviv 69978, Israel} \par
\noindent  \textit{Emails:} \texttt{eleanora@mail.tau.ac.il, asafnach@tauex.tau.ac.il, matanshalev@mail.tau.ac.il} \par
  \addvspace{\medskipamount}

\end{document}